\documentclass[final,3p]{elsarticle}
\usepackage{amssymb}
\usepackage{amsmath}
\usepackage[all,cmtip]{xy}
\usepackage{latexsym}
\usepackage{amsthm}
\usepackage{a4wide}
\usepackage{color}
\usepackage{hyperref}
\usepackage{stmaryrd}
\usepackage{yhmath}
\usepackage{graphicx}
\input diagxy

\usepackage[all,cmtip]{xy}
\input diagxy

\journal{}

\theoremstyle{plain}
  \newtheorem{thm}{Theorem}[section]
  \newtheorem{lem}[thm]{Lemma}
  \newtheorem{prop}[thm]{Proposition}
  \newtheorem{cor}[thm]{Corollary}
\theoremstyle{definition}

  \newtheorem{exmp}[thm]{Example}
  \newtheorem{rem}[thm]{Remark}

\DeclareMathOperator{\ob}{ob}

\def\oto{{\bfig\morphism<180,0>[\mkern-4mu`\mkern-4mu;]\place(86,0)[\circ]\efig}}
\def\rto{{\bfig\morphism<180,0>[\mkern-4mu`\mkern-4mu;]\place(82,0)[\mapstochar]\efig}}

\newcommand{\lra}{\longrightarrow}
\newcommand{\lda}{\swarrow}
\newcommand{\rda}{\searrow}

\newcommand{\Lra}{\Longrightarrow}

\newcommand{\bv}{\bigvee}
\newcommand{\bw}{\bigwedge}

\newcommand{\dv}{\dashv}

\newcommand{\od}{\odot}
\newcommand{\opl}{\oplus}
\newcommand{\ola}{\overleftarrow}
\newcommand{\ora}{\overrightarrow}

\renewcommand{\phi}{\varphi}

\newcommand{\ep}{\varepsilon}

\newcommand{\lam}{\lambda}

\newcommand{\si}{\sigma}

\newcommand{\CA}{\mathcal{A}}
\newcommand{\CB}{\mathcal{B}}

\newcommand{\CD}{\mathcal{D}}

\newcommand{\CH}{\mathcal{H}}

\newcommand{\CP}{\mathcal{P}}
\newcommand{\CQ}{\mathcal{Q}}
\newcommand{\CS}{\mathcal{S}}
\newcommand{\CT}{\mathcal{T}}
\newcommand{\CU}{\mathcal{U}}
\newcommand{\CV}{\mathcal{V}}

\newcommand{\sP}{{\sf P}}

\newcommand{\se}{{\sf e}}
\newcommand{\sm}{{\sf m}}

\newcommand{\sfs}{{\sf s}}
\newcommand{\sy}{{\sf y}}
\newcommand{\sh}{\mathsf{h}}

\newcommand{\bbI}{\mathbb{I}}

\newcommand{\bbS}{\mathbb{S}}
\newcommand{\bbT}{\mathbb{T}}

\newcommand{\Fs}{\mathfrak{s}}
\newcommand{\Fx}{\mathfrak{x}}
\newcommand{\Fy}{\mathfrak{y}}
\newcommand{\Fz}{\mathfrak{z}}

\newcommand{\Cat}{{\bf Cat}}

\newcommand{\Dist}{{\bf Dist}}

\newcommand{\Rel}{{\bf Rel}}
\newcommand{\Set}{{\bf Set}}
\newcommand{\Sup}{{\bf Sup}}
\newcommand{\QCat}{\CQ\text{-}\Cat}

\newcommand{\QDist}{\CQ\text{-}\Dist}

\newcommand{\QRel}{\CQ\text{-}\Rel}

\newcommand{\VRel}{\CV\text{-}\Rel}

\newcommand{\CPd}{\CP^{\dag}}

\newcommand{\syd}{\sy^{\dag}}

\newcommand{\CHd}{{\CH^\dagger}}
\newcommand{\co}{{\rm co}}
\newcommand{\op}{{\rm op}}

\newcommand{\hT}{\hat{T}}
\newcommand{\hCT}{\hat{\CT}}

\newcommand{\ophi}{\ora{\phi}}

\newcommand{\olphi}{\ola{\phi}}

\newcommand{\ssd}{\sfs^{\dag}}
\newcommand{\Fsd}{\Fs^{\dag}}
\newcommand{\Fyd}{\Fy^{\dag}}

\renewcommand{\leq}{\leqslant}
\renewcommand{\geq}{\geqslant}

\newcommand{\QSet}{\Set/\CQ_0}
\newcommand{\ovl}{\overline}

\numberwithin{equation}{section}

\begin{document}

\begin{frontmatter}

\title{Monads on $\mathcal{Q}\text{-}\mathbf{Cat}$ and their lax extensions to $\QDist$}

\author[S]{Hongliang Lai}
\ead{hllai@scu.edu.cn}

\author[Y]{Walter Tholen\corref{cor}}
\ead{tholen@mathstat.yorku.ca}

\cortext[cor]{Corresponding author.}
\address[S]{School of Mathematics, Sichuan University, Chengdu 610064, China}
\address[Y]{Department of Mathematics and Statistics, York University, Toronto, Ontario M3J 1P3, Canada}

\begin{abstract} 
For a small quantaloid $\CQ$, we consider 2-monads on the 2-category $\QCat$ and their lax extensions to the 2-category
$\QDist$ of small $\CQ$-categories and their distributors, in particular those lax extensions that are flat, in the sense that they map identity distributors to identity distributors. In fact, unlike in the discrete case, a 2-monad on $\QCat$ may admit only one flat lax extension.
Every ordinary monad on the comma category $\Set/{\rm ob}\CQ$ with a lax extension to $\CQ\text{-}\bf{Rel}$ gives rise to such a 2-monad on $\QCat$, and we describe this process globally as a coreflective embedding. The $\CQ$-presheaf and the double $\CQ$-presheaf monads are important examples of 2-monads  on $\QCat$ allowing flat lax extensions to $\QDist$, and so are their submonads, obtained by the restriction to conical (co)presheaves and known as the  $\CQ$-Hausdorff and double $\CQ$-Hausdorff monads, which we define here in full generality, thus generalizing some previous work in the case when $\CQ$ is a quantale, or just the ``metric"  quantale $[0,\infty]$. Their discretization leads naturally to various lax extensions of the relevant $\Set$-monads used in monoidal topology.

\end{abstract}

\begin{keyword}
quantaloid \sep $\CQ$-category \sep $\CQ$-distributor \sep 2-monad\sep lax extension \sep flat lax extension \sep presheaf monad \sep double presheaf monad \sep Hausdorff monad

MSC[2010]: 18D20 \sep 18C15 \sep 18A40 \sep 06F99
\end{keyword}

\end{frontmatter}

\section{Introduction}
In order to extend Manes' \cite{Manes1969} description of compact Hausfdorff spaces via ultrafilter convergence to all topological spaces, Barr \cite{Barr1970} employed the prototype of a lax extension of a monad on $\Set$. Observing that a relation $R\subseteq X\times Y$, when written as a morphism $r:X\rto Y$ in the category $\Rel$ of sets and relations, may be factored as
\[r=q_{\circ}\circ p^{\circ},\]
where $p^{\circ}:X\rto R$ is the converse of the graph of the first projection $p:R\to X$ and $q_{\circ}:R\rto Y$ the graph of the second projection $q:R\to Y$,
he extended the ultrafilter functor $\beta:\Set\to\Set$ to $\Rel$ by putting
\[\ovl{\beta}r=(\beta q)_{\circ}\circ(\beta p)^{\circ}:(\beta X\rto \beta R\rto \beta Y)\quad\quad\quad(*)\]
or, in pointwise terms,
\[\Fx(\ovl{\beta}r)\Fy\iff\exists\Fz\in \beta R\,(p[\Fz]=\Fx\,\,\text{and}\,\,q[\Fz]=\Fy),\]
for all ultrafilters $\Fx$ on $X$ and $\Fy$ on $Y$; here $p[\Fz]=\beta p(\Fz)$ denotes the image of the ultrafilter $\Fz$ under the map $p$.
While Barr's extension procedure works well for $\beta$, its use beyond its original purpose seems to be rather limited. For example, when one trades $\beta$ for an arbitrary $\Set$-functor $T$, the procedure yields a lax extension of $T$ to $\Rel$ (if and) only if $T$ preserves weak pullback diagrams \cite{Hofmann2014}. Generally speaking, there is a price to pay for the reliance on the Cartesian structure of $\Set$ when presenting the relation $r:X\rto Y$ as a subobject of $X\times Y$.
% applied to the filter monad, rather than to the ultrafilter monad, the corresponding lax algebras fail to describe all (or any interesting class of) topological spaces, even though another naturally obtained lax extension of the filter monad will do the job, with the same two axioms as those appearing in Barr's description: see Seal \cite{Seal 2004}.
 
 However, there is another relational factorization which, as we will show in this paper, while giving an alternative description of the Barr extension when applied to $\beta$, lends itself to a vast generalization to other endofunctors or monads, reaching considerably beyond the $\Set$-environment. It takes advantage of the fact that the $\Set$-functor $\beta$ is just the ``discrete restriction" of the functor 
$\beta:{\bf Ord}\to{\bf Ord}$ which, for a (pre)ordered set $(X,\leq)$, orders the set of ultrafilters on the set $X$ by
\[\Fx\leq\Fy\iff \forall A\in\Fx\,\forall B\in\Fy\,\exists x\in A\,\exists y\in B\, (x\leq y).\]
With this fact in mind, one factors the relation $r:X\rto Y$ of (discretely ordered) sets through the powerset $PX$, ordered by inclusion, as
\[r=(\ola r)^{\star}\circ (\sy_X)_{\star},\]
with the maps $\sy_X:X\to PX, x\mapsto\{x\}$, and $\ola{r}:Y\to PX,y\mapsto\{x\in X\,|\,x\,r\,y\}$; here, for any 
monotone map $f:(A,\leq)\to(B,\leq)$, we denote by $f_{\star}$ the order graph $\{(fx,y)\,|\,x\in A, y\in B, fx\leq y\}$ of $f$, while $f^{\star}=\{(y,fx)\,|\,x\in A, y\in B, y\leq fx\}$. Now the lax extension $\ovl{\beta}$ of the $\Set$-functor $\beta$ may be described by
\[\ovl{\beta}r= (\beta\ola{r})^{\star}\circ(\beta\sy_X)_{\star}:(\beta X\rto \beta PX\rto \beta Y)\quad\quad\quad(**)\]
or, in pointwise terms,
\[\Fx(\ovl{\beta}r)\Fy\iff\sy_X[\Fx]\leq\ola{r}[\Fy]\iff\forall A\in\Fx\,\forall B\in\Fy\,\exists x\in A\,\exists y\in B\,(x\,r\,y),\]
for all $\Fx\in \beta X,\Fy\in \beta Y$.

The extension procedure $(**)$ was introduced in \cite{Akhvlediani2010} more generally for any endofunctor $T$ of $\CV$-$\Cat$ in lieu of $\beta$, to yield a lax extension of $T$ to the $2$-category $\CV$-$\bf{Dist}$ of small $\CV$-categories and their distributors (=(bi-)modules=profunctors), for any commutative quantale $\CV$. When restricted to $\CV$-relations (= ``discrete" $\CV$-distributors), and for $\CV=2$ the two-element chain, so that $\CV$-$\Cat=\bf{Ord}$, it corresponds to $(**)$. Its superiority over $(*)$ in terms of applicability is founded in the use of the presheaf category of a $\CV$-category $X$ which, when $\CV=2$, reduces to the powerset $PX$ in the discrete case, rather than the use of the Cartesian product in $\Set$.

The question remains how one may guarantee a sufficient supply of endofunctors of $\CV$-$\Cat$ that can cover the relevant applications. One answer to this question was given in \cite{Tholen2009} where it was observed that any monad on $\Set$ equipped with a lax extension to $\CV$-$\Rel$ (the discrete ``skeleton" of $\CV$-$\Dist$) gives rise to a monad on $\CV$-$\Cat$ in a natural way (see also \cite{Hofmann2014}). Another answer was presented in \cite{Lai2016}, in the considerably generalized environment in which the quantale $\CV$ is traded for a small quantaloid $\CQ$ (without any commutativity constraints). There we presented lax extensions to $\QDist$ of the presheaf monad and the copresheaf monad on $\QCat$, as well as their two composite monads. Neither of these four monads may be obtained from a ``discrete" counterpart.

In this paper we merge the aspects of \cite{Tholen2009, Akhvlediani2010, Lai2016} alluded to earlier and present a more comprehensive theory of $2$-monads on $\QCat$ and their lax extensions to $\QDist$, for any small quantaloid $\CQ$. A simple but pivotal observation that was missed in previous works is that a $2$-functor $\CT$ on $\QCat$ may admit only one lax extension $\hCT$ that is {\em flat}, so that $\hCT$ maps identities to identities in $\QDist$. 
Such lax extension must necessarily be minimal, and there is an easy criterion for the extension procedure $(**)$ applied to $\CT$ (in lieu of $\beta$) to produce the minimal flat lax extension of $\CT$: the $2$-functor $\CT$ must preserve the {\em full fidelity of the Yoneda $\CQ$-functor} of every 
$\CQ$-category into its presheaf category.  The $2$-monads on $\QCat$ whose endofunctor satisfies this criterion form the full subcategory $\sf{ffYMnd}(\QCat)$ of $\sf{Mnd}(\QCat)$.

By contrast, in the discrete environment, a $\Set/\CQ_0$-monad (with $\CQ_0={\rm ob}\CQ$) may allow several flat lax extensions to $\CQ$-$\bf{Rel}$. Nevertheless, after the preliminaries of Section 2 we start the essence of this paper in Section 3 in the discrete setting, by showing how every monad on $\QSet$ equipped with a lax extension to $\QRel$ gives rise to a $2$-monad on $\QCat$ which comes equipped with a lax extension to $\QDist$. The lax extensions obtained turn out to be flat, so that this process describes a full embedding of the meta-category 
$\mathsf{ExtMnd}(\QSet)$ of laxly extended $\QSet$-monads into $\sf{ffYMnd}(\QCat)$. One of the two main results of this paper says that this embedding is coreflective, with the coreflector given by ``discretization" (Theorem \ref{coref}). 

The other main insight is given in Section 5 where we show that the principal lax monad extensions discussed in \cite{Hofmann2014} all arise from an application of the aforementioned coreflector to one of the four monads on $\QCat$ discussed in \cite{Lai2016}, given by the presheaf and copresheaf monads on $\QCat$ and their two composite monads, as well as by their restrictions to {\em conical} (co)presheaves. As observed by \cite{Stubbe2010} for the presheaf monad, such restriction yields the {\em Hausdorff monad}, as introduced in \cite{Akhvlediani2010} in order to provide a categorical environment for the study of Hausdorff and Gromov distances. The restriction of the double (co)presheaf monads to conical (co)presheaves is, however, not trivial and requires us to impose the condition that the ambient quantaloid $\CQ$ be completely distributive, so that all hom sets of $\CQ$ are completely distributive lattices: see Theorem \ref{doubleHaus}.

\section{Quantaloid-enriched categories and their distributors}

A \emph{quantaloid} \cite{Rosenthal1996} is a category enriched in the monoidal-closed category $\Sup$ \cite{Joyal1984} of complete lattices and $\sup$-preserving maps. Explicitly, a quantaloid $\CQ$ is a 2-category with the 2-cells given by the order of morphisms, denoted by $\leq$, such that each hom-set $\CQ(p,q)$ is a complete lattice and the composition of morphisms from either side preserves arbitrary suprema. Hence, $\CQ$ has ``internal homs'', denoted by $\lda$ and $\rda\,$, as the right adjoints of the composition maps
$$(-)\circ \alpha\dv (-)\lda \alpha:\CQ(p,r)\to\CQ(q,r)\quad\text{and}\quad \beta\circ (-)\dv \beta\rda (-):\CQ(p,r)\to\CQ(p,q);$$
that is, for all morphisms $\alpha:p\to q$, $\beta:q\to r$, $\gamma:p\to r$ in $\CQ$, one has the equivalences
$$\alpha\leq \beta\rda \gamma\iff \beta\circ\alpha\leq\gamma\iff \beta\leq \gamma\lda \alpha.$$

Throughout this paper, we let $\CQ$ be a \emph{small} quantaloid. From $\CQ$ one forms a new (large) quantaloid $\QRel$ of \emph{$\CQ$-relations}, with the following data: its objects are those of $\Set/\CQ_0$, with $\CQ_0:=\ob\CQ$, {\em i.e.}, sets $X$ equipped with an \emph{array} (or \emph{type}) map $|\text{-}|:X\to\CQ_0$, and a morphism $r:X\rto Y$ in $\QRel$ is given by a map that assigns to each pair $x\in X$, $y\in Y$ a morphism $r(x,y):|x|\to|y|$ in $\CQ$; its composite with $s:Y\rto Z$ is defined by
$$(s\circ r)(x,z)=\bv_{y\in Y}s(y,z)\circ r(x,y),$$
and $1^\circ_X:X\rto X$ with
$$1^\circ_X(x,y)=\begin{cases}
1_{|x|} & \text{if}\ x=y,\\
\bot & \text{else,}
\end{cases}$$
acts as the identity morphism on $X$. As $\CQ$-relations are equipped with the pointwise order inherited from $\CQ$, internal homs in $\QRel$ are computed pointwise by
$$(t\lda r)(y,z)=\bw_{x\in X}t(x,z)\lda r(x,y)\quad\text{and}\quad(s\rda t)(x,y)=\bw_{z\in Z}s(y,z)\rda (x,z)t,$$
for all $r:X\rto Y$, $si:Y\rto Z$, $t:X\rto Z$.

A map $f:X\to Y$ in $\QSet$ may be seen as a $\CQ$-relation via its graph $f_\circ$ or its cograph $f^\circ$,
which are given by
\begin{equation} \label{graph_in_Rel}
\begin{array}{ll}
f_\circ:X\rto Y,& f_\circ(x,y)=1^\circ_Y(fx,y),\\
f^\circ:Y\rto X,& f^\circ(y,x)=1^\circ_Y(y,fx).
\end{array}
\end{equation}

A (small) \emph{$\CQ$-category} is precisely an (internal) monad in the 2-category $\QRel$; or, equivalently, a monoid in the monoidal-closed category $(\QRel(X,X),\circ)$
%\footnote{$(\QRel(X,X),\circ)$ is in fact a \emph{quantale} \cite{Rosenthal1990}, i.e., a one-object quantaloid.}
for some $X$ over $\CQ_0$. Explicitly, a $\CQ$-category consists of an object $X$ in $\Set/\CQ_0$ and a $\CQ$-relation $a:X\rto X$ such that $1^\circ_X\leq a$ and $a\circ a\leq a$; that is:
\[1_{|x|}\leq a(x,x)\quad\text{and}\quad a(y,z)\circ a(x,y)\leq a(x,z)\]
for all $x,y,z\in X$. For every $\CQ$-category $(X,a)$, the underlying (pre)order on $X$ is given by
$$x\leq x'\iff|x|=|x'|\ \text{and}\ 1_{|x|}\leq a(x,x').$$

A map $f:(X,a)\to(Y,b)$ between $\CQ$-categories is a \emph{$\CQ$-functor} if it lives in $\Set/\CQ_0$ and satisfies (any of) the following four equivalent conditions:
\begin{enumerate}[(i)]
\item $f_\circ\circ a\leq b\circ f_\circ$;
\item $a\circ f^{\circ}\leq f^{\circ}\circ b$;
\item $a\leq f^\circ\circ b\circ f_\circ$;
\item $\forall x,x'\in X:\,a(x,x')\leq b(fx,fx')$.
 \end{enumerate}
 With the pointwise order of $\CQ$-functors inherited from $Y$, i.e.,
\begin{align*}f\leq g:(X,a)\to(Y,b)&\iff\forall x\in X:\ fx\leq gx\\
&\iff\forall x\in X:\ 1_{|x|}\leq b(fx,gx),
\end{align*}
$\CQ$-categories and $\CQ$-functors are organized into a 2-category, denoted by $\QCat$.

A $\CQ$-relation $\phi:X\rto Y$ becomes a \emph{$\CQ$-distributor} $\phi:(X,a)\oto(Y,b)$ if it is compatible with the $\CQ$-categorical structures $a$ and $b$, in the sense that
$$b\circ\phi\circ a\leq\phi,$$
in which case one actually has $b\circ\phi\circ a=\phi.$ With the composition and internal homs calculated in the same way as for $\CQ$-relations,
$\CQ$-categories and $\CQ$-distributors constitute a quantaloid, denoted by $\QDist$; the identity $\CQ$-distributor on $(X,a)$ is given by the hom $a:(X,a)\oto(X,a)$.

The set-of-objects functor $$o:\QCat\to\Set/\CQ_0,\quad (X,a)\mapsto X,$$ has a left adjoint $$d:\Set/\CQ_0\to\QCat,\quad X\mapsto(X,1^\circ_X),$$ which embeds $\Set/\CQ_0$ into $\QCat$ as a full coreflective subcategory. Furthermore, with no change on their effect on objects, the functors $o$ and $d$ may be extended to $\QRel$ and $\QDist$ in an obvious way, thus yielding respectively a lax $2$-functor (not a functor)  $$\bf o:\QDist\to\QRel$$ and a $2$-embedding 
$$\bf d:\QRel\to\QDist.$$ 

Every $\CQ$-functor $f:(X,a)\to(Y,b)$ induces an adjunction
$$f_*\dv f^*:(Y,b)\oto(X,a)$$ in the $2$-category $\QDist$ with
\begin{equation} \label{graph_def}
f_*=b\circ f_\circ,\quad f^*=f^\circ\circ b.
\end{equation}
Hence, there are $2$-functors $(-)_*:\QCat^{\rm co}\to\QDist$ and $(-)^*:\QCat^{\rm op}\to\QDist$ which map objects identically and make
%called respectively the \emph{graph} and \emph{cograph} of $f$ in $\QDist$.
the diagrams
\[\begin{array}{cc}\bfig \square<800,500>[\QSet`\QRel`\QCat^{\rm co}`\QDist;(-)_\circ`d`\bf d`(-)_*]\efig & \bfig\square<800,500>[(\QSet)^{\rm op}`\QRel`\QCat^{\rm op}`\QDist;(-)^\circ`d^{\rm{op}}`\bf d`(-)^*]\efig
 \end{array}\]
commute (where ``$\co$'' refers to the dualization of 2-cells).
Note that $$a=(1_X)_*=1_X^*,$$ for every $\CQ$-category $X=(X,a)$. A $\CQ$-functor $f:X\to Y$ is {\em fully faithful} if $f^*\circ f_*=1^*_X$; that is, if $a(x,y)=b(fx,fy)$ for all $x,y\in X$.

With $\{s\}$ denoting the singleton $\CQ$-category with only one object $s\in\CQ_0$ of array $s$ and hom $1_{\{s\}}^*(s,s)=1_s$, $\CQ$-distributors of the form $\si:X\oto\{s\}$, called \emph{presheaves} on $X$; with $|\si|=s$ and
$1_{\CP X}^*(\si,\tau)=\tau\lda\si$ they constitute the $\CQ$-category $\CP X$. Dually, the \emph{copresheaf} $\CQ$-category $\CP^\dagger X$ consists of $\CQ$-distributors $\si:\{s\}\oto X$ with $1_{\CP^\dagger X}^*(\si,\tau)=\tau\rda\si$.

%\begin{rem} \label{PdX_QDist_order}
%For any $\CQ$-category $X$, it follows from the definition that the underlying order on $\PdX$ is the \emph{reverse} local order in $\QDist$, i.e.,
%$$\tau\leq\tau'\ \text{in}\ \PdX\iff\tau'\leq\tau\ \text{in}\ \QDist.$$
%That is why we use a different symbol ``$\leq$'' for the underlying order of $\CQ$-categories and the 2-cells in $\QCat$, while the 2-cells in $\CQ$ and $\QDist$ are denoted by ``$\leq$''.
%\end{rem}

%A $\CQ$-category $X$ is \emph{complete} if the \emph{Yoneda embedding}
%$$\sy_X:X\to\PX,\quad x\mapsto 1_X^*(-,x)$$
%has a left adjoint $\sup_X:\PX\to X$ in $\QCat$; that is,
%$$1_X^*({\sup}_X\si,-)=1_{\PX}^*(\si,\sy_X-)=1_X^*\lda\si$$
%for all $\si\in\PX$. It is well known that $X$ is a complete $\CQ$-category if and only if $X^{\op}:=(X,(1_X^*)^{\op})$ with $(1_X^*)^{\op}(x,x')=1_X^*(x',x)$ is a complete $\CQ^{\op}$-category \cite{Stubbe2005}, where the completeness of $X^{\op}$ may be translated as the \emph{co-Yoneda embedding}
%$$\syd_X:X\to\PdX,\quad x\mapsto 1_X^*(x,-)$$
%admitting a right adjoint $\inf_X:\PdX\to X$ in $\QCat$.

%\begin{lem} \label{Yoneda} {\rm\cite{Shen2013a,Stubbe2005}} (Yoneda %lemma) Let $X$ be a $\CQ$-category. For any $\si\in\PX$, $\tau\in\PdX$,
 %   $$\si=(\sy_X)_*(-,\si)=1_{\PX}^*(\sy_X-,\si)\quad\text{and}\quad\tau=(\syd_X)^*(\tau,-)=1_{\PdX}^*(\tau,\syd_X-).$$
 %   In particular, both $\sy_X:X\to\PX$ and $\syd_X:X\to\PdX$ are fully faithful.
%\end{lem}

Every $\CQ$-distributor $\phi:X\oto Y$ induces $\CQ$-functors  \cite{Shen2013a}  given by
\begin{equation}\phi^\odot:\CP Y\to\CP X,\,\phi^\odot\tau=\tau\circ\phi, \text{ and }
\phi^\oplus:\CPd X\to\CPd Y,\,\phi^\oplus\si=\phi\circ\si.
\end{equation}
One obtains two pairs of adjoint 2-functors \cite{Heymans2010}, described by
\begin{equation} \label{QCat_QDist_adjunction}
\bfig
\morphism(-150,100)<400,0>[X`Y;\phi] \place(40,100)[\circ]
\morphism(-150,-100)<400,0>[Y`\CP X;\olphi]
\morphism(-300,20)/-/<700,0>[`;]
\place(820,20)[(\olphi y)_x=\phi(x,y)]
\morphism(1600,0)/@{->}@<7pt>/<800,0>[\QCat`(\QDist)^{\op},;(-)^*]
\morphism(2400,0)|b|/@{->}@<5pt>/<-800,0>[(\QDist)^{\op},`\QCat;\CP]
\place(1950,0)[\bot]
\place(1500,-250)[(\phi^{\od}:\CP Y\to\CP X)]
\morphism(2050,-250)/|->/<-150,0>[`;]
\place(2400,-250)[(\phi:X\oto Y)]
\morphism(-150,-500)<400,0>[X`Y;\phi] \place(40,-500)[\circ]
\morphism(-150,-700)<400,0>[X`\CPd Y;\ophi]
\morphism(-300,-580)/-/<700,0>[`;]
\place(820,-580)[(\ophi x)_y=\phi(x,y)]
\morphism(1600,-600)/@{->}@<7pt>/<800,0>[\QCat`(\QDist)^{\co},;(-)_*]
\morphism(2400,-600)|b|/@{->}@<5pt>/<-800,0>[(\QDist)^{\co},`\QCat;\CPd]
\place(1950,-600)[\bot]
\place(1450,-850)[(\phi^{\opl}:\CPd X\to\CPd Y)]
\morphism(2050,-850)/|->/<-150,0>[`;]
\place(2400,-850)[(\phi:X\oto Y)]
\efig
\end{equation}
 %The unit $\sy$ and the counit $\ep$ of the adjunction $(-)^*\dv\sP$ are respectively given by the Yoneda embeddings and their graphs
%$$\ep_X:=(\sy_X)_*:X\oto\PX.$$
The endofunctor of the \emph{presheaf 2-monad} $\frak{P}=(\CP,\sfs,\sy)$ on $\QCat$ induced by $(-)^*\dv\CP$ sends each $\CQ$-functor $f:X\to Y$ to
$$f_!:=(f^*)^{\od}:\CP X\to\CP Y;$$
%which admits a right adjoint $f^!:=(f^*)_{\od}=(f_*)^{\od}:\PY\to\PX$ in $\QCat$;
the unit $\sy$ is given by the Yoneda functor
\begin{equation}
\sy_X=\overleftarrow{1^*_X}:X\to\CP X,
\end{equation}
and the monad multiplication $\sfs$ at $X$ is obtained from the counit $(\sy_X)_*$ as
\begin{equation} \label{s_def}
\sfs_X=(\sy_X)_*^{\od}:\CP\CP X\to\CP X.
\end{equation}
Similarly, the endofunctor of the induced \emph{copresheaf 2-monad} $\frak{P}^{\dag}=(\CP^\dagger,\ssd,\syd)$ on $\QCat$ sends $f$ to
$$(f_*)^{\opl}:\CP^\dagger X\to\CP^\dagger Y,$$
with the unit $\sy^\dagger$ given by the co-Yoneda functor
 \begin{equation}
 \sy^\dagger_X=\overrightarrow{1^*_X}:X\to\CP^\dagger Y,
 \end{equation}
 and the monad multiplication $\ssd$ by
\begin{equation} \label{sd_def}
\ssd_X=((\syd_X)^*)^{\oplus}:\CP^\dagger\CP^\dagger X\to\CP^\dagger X.
\end{equation}

%\begin{lem} \label{functor_order}
%For all $\CQ$-functors $f,g:X\to Y$ and $\CQ$-distributors $\phi,\psi:X\oto Y$,
%\begin{enumerate}[label={\rm (\arabic*)}]
%\item \label{f_leq_g}
%$f\leq g\iff f_*\succeq g_*\iff f^*\leq g^*\iff f_!\leq g_!\iff f_{\ie}\leq g_{\ie}\iff f^!\geq g^!\iff f^{\ie}\geq g^{\ie}.$
%\item \label{phi_leq_psi}
%$\phi\leq\psi\iff\phi^{\od}\leq\psi^{\od}\iff\phi^{\opl}\geq\psi^{\opl}\iff\olphi\leq\olpsi\iff\ophi\geq\opsi$.
%\end{enumerate}
%\end{lem}

\section{From laxly extended monads on $\QSet$ to laxly extended monads on $\QCat$}

First we recall the ``discrete version" of lax extensions of monads (see \cite{Tholen2016}), by considering a mere functor
$T:\QSet\to\QSet$. A {\em lax extension of $T$ to} $\QRel$ is a lax $2$-functor  $\hT:\QRel\to\QRel$ which agrees with $T$ on objects and satisfies the lax extension conditions
\begin{equation}\label{set_ext_con}(Tf)_\circ\leq\hT(f_\circ),\quad (Tf)^\circ\leq\hT(f^\circ)\end{equation}
for all maps $f:X\to Y$ in $\QSet$. The lax extension is {\em flat} if these inequalities are in fact equalities.

The lax 2-functor $\hT$ is a {\em lax extension of the $\QSet$-monad $\bbT=(T,m,e)$ to} $\QRel$ if $\hT$ is a lax extension of $T$ which makes both $m_{\circ}:\hT\hT\to\hT$ and $e_{\circ}:1\to\hT$ oplax, {\em i.e.}, which satisfies
\begin{equation}(m_Y)_\circ\circ\hT\hT r\leq\hT r\circ(m_X)_\circ,\quad
(e_Y)_\circ\circ r\leq \hT r\circ(e_X)_\circ\end{equation}  
or, equivalently,
\begin{equation}\label{mon_ext_Rel}\hT\hT r\leq(m_Y)^\circ\circ\hT r\circ(m_X)_\circ,\quad
r\leq(e_Y)^\circ\circ\hT r\circ(e_X)_\circ,\end{equation}
for all $\CQ$-relations $r:X\rto Y$.

\begin{prop}\label{discrete whisker}
{\rm (\cite{Tholen2016, Hofmann2014})} 
Given $T:\QSet\to\QSet$,
a lax 2-functor $\hT:\QRel\to\QRel$ which coincides with $T$ on objects is a lax extension of $T$ if, and only if,
\begin{equation*}
\hT(f^\circ\circ r)=(T f)^\circ\circ\hT r\text{ and }\hT(s\circ f_\circ)=\hT s\circ(Tf)_\circ.
\end{equation*}
 for all maps $f:X\to Y$ over $\CQ_0$ and all $\CQ$-relations $r:Z\rto Y$, $s:Y\rto Z$. In this case, $\hT$ is flat if, and only if, for all sets $X$ over $\CQ_0$, \[\hT1_X^{\circ}=1_{TX}^{\circ}.\]
\end{prop}

For monads $\bbS$ and $\bbT$ on $\QSet$, a morphism of lax extensions $\lam:(\bbS,\hat{S})\to(\bbT,\hat{T})$ is a monad morphism $\lam:\bbS\to\bbT$ such that $\lam_{\circ}$ becomes an oplax transformation from $\hat{S}$ to $\hat{T}$, {\em i.e.}, 
\[(\lam_Y)_\circ\circ\hat{S}r\leq\hat{T}r\circ(\lam_X)_\circ,\]
for all $\CQ$-relations $r:X\rto Y$.

The conglomerate of monads on $\QSet$ and their morphisms constitute the metacategory $\mathsf{Mnd}(\QSet)$. The $\QSet$-monads that come equipped with a lax extensions of monads to $\QRel$ and their morphisms form the metacategory $\mathsf{ExtMnd(\QSet})$. One easily proves the following Proposition, similarly to the proof of its non-discrete counterpart that appears below as Proposition \ref{top_cat}.

\begin{prop} $\mathsf{ExtMnd}(\QSet)$ is topological over $\mathsf{Mnd}(\QSet)$.
\end{prop}

Lax extensions of $\QCat$-monads may be considered analogously to their discrete prototypes (see \cite{LST}). Given a 2-functor $\CT:\QCat\to\QCat$, a \emph{lax extension of $\CT$ to $\QDist$} is a lax functor
$$\hCT:\QDist\to\QDist$$
which coincides with $\CT$ on objects and satisfies the lax extension conditions
\begin{equation}\label{cat_ext_con}
    (\CT f)_*\leq\hCT(f_*),\quad (\CT f)^*\leq\hCT(f^*),
\end{equation}
for all $\CQ$-functors $f:X\to Y$. If these inequalities are in fact equalities, we say that $\hCT$ is {\em flat}. By definition then, a lax extension $\hCT$ of $\CT$ is flat if, and only if, it makes the diagrams
 \[\begin{array}{cc}\bfig \square<800,500>[\QCat^{\rm co}`\QCat^{\rm co}`\QDist`\QDist;\CT`(-)_*`(-)_*`\hCT]\efig & \bfig\square<800,500>[\QCat^{\rm op}`\QCat^{\rm op}`\QDist`\QDist;\CT^{\rm op}`(-)^*`(-)^*`\hCT]\efig
 \end{array}\]
 commute.
 
Given a $2$-monad $\frak{T}=(\CT,\sm,\se)$ on $\QCat$, $\hCT$ is a {\em lax extension of the monad} $\frak{T}$ if it is a lax extension of the 2-functor $\CT$ which makes both $\sm_*:\hCT\hCT\to\hCT$ and $\se_*:1\to\hCT$ oplax, {\em i.e.},
\begin{equation}(\sm_Y)_*\circ\hCT\hCT \varphi\leq\hCT \varphi\circ(\sm_X)_*,\quad
(\se_Y)_*\circ\varphi\leq \hCT\varphi\circ(\se_X)_*\end{equation}
or, equivalently,
\begin{equation}\label{ext_of_2-mon}\hCT\hCT\varphi\leq(\sm_Y)^*\circ\hCT \varphi\circ(\sm_X)_*,\quad
\varphi\leq(\se_Y)^*\circ\hCT \varphi\circ(\se_X)_*,
\end{equation}
for all $\CQ$-distributors $\varphi:X\oto Y$.

It is useful to describe the extension conditions (\ref{cat_ext_con}) equivalently by the {\em left- and right-whiskering} properties, as follows.

\begin{prop}\label{hT_graph}  
For a 2-functor $\CT:\QCat\to\QCat$,
a lax 2-functor $\hCT:\QDist\to\QDist$ that agrees with $\CT$ on objects is a lax extension of $\CT$ if, and only if, 
\begin{equation*}
    \hCT(f^*\circ\phi)=(\CT f)^*\circ\hCT\phi\text{ and }\hCT(\psi\circ f_*)=\hCT\psi\circ(\CT f)_*,
\end{equation*}
for all $\CQ$-functors $f:X\to Y$ and all $\CQ$-distributors $\phi:Z\oto Y$, $\psi:Y\oto Z$. In this case, $\hCT$ is flat
if, and only if, \[\hCT1_X^*=1_{\CT X}^*,\] for all $\CQ$-categories $X$.
\end{prop}

\begin{proof} 
The whiskering conditions give immediately \[\hCT f_*=\hCT(1^*_Y\circ f_*)=\hCT 1^*_X\circ(\CT f)_*\geq(\CT f)_*\] and \[\hCT f^*=\hCT(f_*\circ 1^*_X)=(\CT f)_*\circ\hCT 1^*_X\geq(\CT f)^*.\] 
Conversely, first notice that, on one hand, one has \[\hCT(f^*\circ\phi)\geq\hCT f^*\circ\hCT\phi\geq(\CT f)^*\circ\hCT\phi,\] and, on the other hand, one obtains 
\begin{align*}\hCT(f^*\circ\phi)&\leq(\CT f)^*\circ(\CT f)_*\circ\hCT(f^*\circ\phi)\\
&\leq(\CT f)^*\circ\hCT f^*\circ\hCT(f^*\circ\phi)\\
&\leq(\CT f)^*\circ\hCT(f_*\circ f^*\circ\phi)\\
&\leq(\CT f)^*\circ\hCT\phi,
\end{align*}
which gives $\hCT(f^*\circ\phi)=(\CT f)^*\circ\phi$, as desired. The other equation may be checked similarly, and the non-trivial part of the additional statement follows from the consideration of $\varphi=1_X^*$.
\end{proof}

%Let $\mathbb{S}=(S,n,d)$ and $\mathbb{T}=(T,m,e)$ are both monads on a category $\mathbf{X}$. A morphism of monads $\lam:\mathbb{S}\to\mathbb{T}$ is a natural transformation $\lam:S\to T$ that preserves the monad structure:
%\[\lam\cdot n=m\cdot(\alpha\bullet\alpha),\quad \alpha\cdot d=e.\]

For $2$-monads $\frak{S}=(\CS,\mathsf{n},\mathsf{d})$ and $\frak{T}=(\CT,\sm,\se)$ on $\QCat$, a morphism of lax extensions $\lam:(\frak{S},\hat{\CS})\to(\frak{T},\hCT)$ is a morphism $\lam:\frak{S}\to\frak{T}$ of monads such that $\lam_*:\hat{\CS}\to\hat{\CT}$ becomes an oplax transformation, {\em i.e.},
\[(\lam_Y)_*\circ\hat{S}\varphi\leq\hat{T}\varphi\circ(\lam_X)_*,\]
for all $\CQ$-distributors $\varphi:X\oto Y$.

With these morphisms the $2$-monads on $\QCat$ which come with a lax extension to $\QDist$ constitute the metacategory $\mathsf{ExtMnd}(\QCat)$. The forgetful functor to the metacategory $\mathsf{Mnd}(\QCat)$ behaves well, as we show next.

\begin{prop}\label{top_cat}
$\mathsf{ExtMnd}(\QCat)$ is topological over $\mathsf{Mnd}(\QCat)$.
\end{prop}
\begin{proof}
One easily verifies that, for a morphism of $2$-monads $\lam:\frak{S}\to\frak{T}$ and a lax extension $(\frak{T},\hCT)$, the {\em initial extension} of $\CS$ induced by $\lam$ is given by \[\hat{\CS}\phi=(\lam_{(Y,b)})^*\circ\hCT\phi\circ(\lam_{(X,a)})_*\]
for all $\CQ$-distributors $\phi:(X,a)\oto(Y,b)$. Thus, for a family of monad morphisms $\lam_i:\frak{S}\to\frak{T}_i$ and a family of lax extensions $(\frak{T}_i,\hCT_i)$, for each $i$ one has the initial extension $(\frak{S},\hat{\CS}_i)$ induced by $\lam_i$. To complete the proof, one needs  to show that the meet $\hat{\CS}$ given by $\hat{\CS}\phi=\bw_i\hat{\CS}_i\phi$ for all $\CQ$-distributors $\phi:(X,a)\oto(Y,b)$ is also a lax extension of the monad $\frak{S}$. The verification of this fact is routine.
\end{proof}

\begin{rem}
As a consequence, given a $2$-monad $\frak{T}$ on $\QCat$, on one hand, there is a largest lax extension $\CT^\top$; it maps each $\CQ$-distributor $\phi:X\oto Y$ to the greatest distributor from $\CT X$ to $\CT Y$. On the other hand, the least lax extension also exists,  which we will describe explicitly in Section 4 if $\CT$ maps every Yoneda functor (2.v) to a fully faithful $\CQ$-functor.
\end{rem}

We will now show that monads on $\QSet$ laxly extended to $\QRel$ induce $2$-monads on $\QCat$ that are laxly extended to $\QDist$. For that, first let $\hT$ be a lax extension of a functor $T:\QSet\to\QSet$. As observed in \cite{Tholen2009} in the quantalic context, given a $\CQ$-category $(X,a)$, since
\[1_{TX}\leq\hT 1_X\leq\hT a,\quad\hT a\circ\hT a\leq\hT(a\circ a)=\hT a,\] $(TX,\hT a)$ is also a $\CQ$-category. If $f:(X,a)\to(Y,b)$ is a $\CQ$-functor, then
\[\hT a\leq\hT(f^\circ\circ b\circ f_\circ)=(T f)^\circ\circ\hT b\circ(T f)_\circ\] shows that $Tf:(TX,\hT a)\to(TY,\hT b)$ is also a $\CQ$-functor. In this way one obtains a $2$-functor \begin{equation}\label{ext_fun}\CT:\QCat\to\QCat,\quad (X,a)\mapsto(TX,\hT a).\end{equation}
Furthermore, if $\bbT=(T,m,e)$ is a monad on $\Set/\CQ_0$ and $\hT$ is a lax extension of $\bbT$ to $\QRel$, oplaxness of $m_{\circ}$ and $e_{\circ}$ with respect to $\hT$ implies that, for every $\CQ$-category $(X,a)$, both $m_X:TTX\to TX$ and $e_X:X\to TX$ become $\CQ$-functors
 \begin{equation}\label{ext_mon}\sm_X:\hT\hT(X,a)\to\hT(X,a), \quad \se_X:(X,a)\to \hT(X, a).\end{equation}
Hence, the monad $\bbT$ on $\QSet$ with its lax extension $\hT$ to $\QRel$ has been shown to ``lift" along $o:\QCat\to\QSet$ to a $2$-monad  $(\CT,\mathsf{m},\mathsf{e})$ on $\QCat$,
in the sense that one has  
\begin{equation}\label{lifting}o\CT=To, \quad o\mathsf{m}=mo, \quad o\mathsf{e}=eo.\end{equation}

We will now verify that $(\CT,\mathsf{m},\mathsf{e})$ comes with a lax extension to $\QDist$ that is obtained from $\hT$ in a straightforward manner.

\begin{prop}\label{ext_ext} For a monad $\mathbb T$ on $\QSet$ with a lax extension $\hT$ to $\QRel$, the $2$-functor $\CT:\QCat\lra\QCat$ defined by {\rm (\ref{ext_fun})} admits a uniquely determined lax extension $\hCT$ to $\QDist$ with ${\bf o}\hCT=\hT{\bf o}$, for the forgetful ${\bf o}:\QDist\to\QRel$; it commutes with $\hT$ also via the ``discrete" ${\bf d}$, as shown by the commutative diagrams
\[
\begin{array}{cc}
\bfig \square<800,500>[\QRel`\QRel`\QDist`\QDist;\hT`{\bf d}`{\bf d}`\hCT] \efig &
\bfig \square<800,500>[\QDist`\QDist`\QRel`\QRel;\hCT`{\bf o}`{\bf o}`\hT] \efig
\end{array}
\]
Furthermore, $\hCT$ is flat, and it yields a lax extension of the monad $\frak{T}=(\CT,\mathsf{m},\mathsf{e})$ to $\QDist$ defined by {\rm (\ref{ext_fun})} and {\rm (\ref{ext_mon})}.
\end{prop}

\begin{proof} The constraint  ${\bf o}\hCT=\hT{\bf o}$ dictates
putting $\hCT\varphi(\Fx,\Fy)=\hT\varphi(\Fx,\Fy)$ for every distributor $\phi:(X,a)\oto(Y,b)$ and all $\Fx\in TX, \Fy\in TY$. To show the existence of $\hCT$, we first confirm that $\hT\phi$ is indeed a $\CQ$-distributor from $(TX,\hT a)$ to $(TY,\hT b)$: 
\[\hT b\circ\hT\phi\circ\hT a\leq\hT(b\circ\phi\circ a)\leq\hT\phi.\]
Second, for every $\CQ$-category $(X,a)$ one has
\[1^*_{\CT (X,a)}=\hT a=\hCT 1^*_{(X,a)},\]
and $\hCT\psi\circ\hCT\phi\leq\hCT(\psi\circ\phi)$ holds for
all distributors $\phi:(X,a)\oto(Y,b)$ and $\psi:(Y,b)\oto(Z,c)$, since
$\hT\psi\circ\hT\phi\leq\hT(\psi\circ\phi)$ holds for their underlying $\CQ$-relations.
Thus, $\hCT$ is a flat lax functor.

Thirdly, for a $\CQ$-functor $f:(X,a)\to(Y,b)$, Proposition \ref{discrete whisker} gives
\[\hCT(f_*)=\hT(b\circ f_\circ)=\hT b\circ(Tf)_\circ=(\CT f)_*,\]
and
\[\hCT(f^*)=\hT(f^\circ \circ b)=(Tf)^\circ\circ\hT b=(\CT f)^*.\]
Hence, $\hCT$ is a lax extension of the $2$-functor $\CT$.

%By Proposition \ref{uni_fla_ext} and Theorem \ref{fla_imp_mon}, $(\hCT,\sm,\se)$ must be the unique lax extension of the monad $(\CT,\sm,\se)$.
Using the fact that $\hT\varphi$ is a $\CQ$-distributor $(TX,\hT a)\oto(TY,\hT b)$ for every $\CQ$-distributor $\varphi:(X,a)\oto(Y,b)$, one checks the oplaxness of $\sm$ and $\se$, as follows:
\[(\se_Y)^*\circ\hCT\varphi\circ(\se_X)_*
=(e_Y)^\circ\circ\hT b\circ\hT\varphi\circ\hT a\circ(e_X)_\circ
=(e_Y)^\circ\circ\hT\varphi\circ(e_X)_\circ
\geq\varphi,\]
\[(\sm_Y)^*\circ\hCT\varphi\circ(\sm_X)_*
=(m_Y)^\circ\circ\hT b\circ\hT\varphi\circ\hT a\circ(m_X)_\circ
=(m_Y)^\circ\circ\hT\varphi\circ(m_X)_\circ
\geq\hT\hT\varphi
=\hCT\hCT\varphi.\] Therefore, $\hCT$ is a lax extension of the $2$-monad $\frak{T}=(\CT,\sm,\se)$.
\end{proof}

\begin{thm}\label{Delta}
The process $(\bbT,\hat{T})\mapsto(\frak{T},\hCT)$ described by Proposition {\rm \ref{ext_ext}} is the object assignment of a full and faithful functor \[\Delta:\mathsf{ExtMnd}(\QSet)\to \mathsf{ExtMnd}(\QCat).\]
\end{thm}
\begin{proof}
For every morphism $\lam:(\bbS,\hat{S})\to(\bbT,\hT)$ in $\mathsf{ExtMnd}(\QSet)$, since $(\lam_X)_\circ\circ\hat{S} a\leq \hT a\circ(\lam_X)_\circ$ for every $\CQ$-category $(X,a)$,
the map $\lam_X$ is in fact a $\CQ$-functor $(SX,\hat{S}a)\to(TX,\hT a)$. Furthermore, since
 \[(\lam_Y)_*\circ\hat{\CS}\phi=(\lam_Y)_\circ\circ\hat{S}\phi
\leq\hat{T}\phi\circ(\lam_X)_\circ=\hat{\CT}\phi\circ(\lam_X)_*\]
for every $\CQ$-distributor $\phi:(X,a)\oto(Y,b)$,
the morphism $\lam:(\bbS,\hat{S})\to(\bbT,\hT)$ can be considered as a morphism $\Delta\lam:(\frak{S},\hat{\CS})=\Delta(\bbS,\hat{S})\to(\frak{T},\hCT)=\Delta(\bbT,\hT)$.

Now consider any morphism $\kappa:(\frak{S},\hat{\CS})\to(\frak{T},\hCT)$ in $\mathsf{ExtMnd}(\QCat)$. Since for
every $\CQ$-category $(X,a)$ one has the $\CQ$-functor $\ep_{(X,a)}=1_X:do(X,a)=(X,1^\circ_X)
\to(X,a)$, naturality of $\kappa:\CS\to\CT$ gives $\kappa_{(X,a)}\cdot\CS 1_X=\CT 1_X\cdot\kappa_{dX}$, so that $o\kappa_{(X,a)}=o\kappa_{dX}$. Hence, putting $\lam_X=o\kappa_{dX}$ by necessity, one easily sees that $\lam:(\bbS,\hat{S})\to(\bbT,\hT)$ is indeed a morphism of lax extensions with $\Delta\lam=\kappa$. Therefore, $\Delta$ is full and faithful.
\end{proof}

We now construct a left inverse to the functor $\Delta$ and consider any $2$-monad $\frak{T}=(\CT,\sm,\se)$ on $\QCat$. One then obtains a monad $\bbT=(T,m,e)$ on $\QSet$ by composing the Eilenberg-Moore adjunction of $\frak{T}$ with the adjunction $d\dashv o:\QCat\to\QSet$, so that
\begin{equation}
\label{res_mon} T=o\CT d, \quad m=o\,\sm\, d\cdot o\CT\varepsilon\CT d,\quad e=o\,\se\, d,
\end{equation}
where $\ep$ is the counit of $d\dashv o$ (as in the proof of Theorem \ref{Delta}). 

\begin{prop}\label{non_dis_to_dis} 
For every $2$-functor $\CT:\QCat\to\QCat$, a lax extension $\hCT$ of $\CT$ to $\QDist$ yields the lax extension $\hT={\bf o}\hCT {\bf d}$ of $T=o\CT d:\Set/\CQ_0\to\Set/\CQ_0$ to $\QRel$. Moreover, if $\CT$ belongs to a $2$-monad $\frak{T}=(\CT,\mathsf{m},\mathsf{e})$, then $\hat{T}$ is a lax extension of $\bbT=(T,m,e)$ as defined by {\rm (\ref{res_mon})}.
\end{prop}

\begin{proof} 
As a composite of lax 2-functors, $\hT$ is also one. The lax extension conditions for $\hT$ follow from
 $$\hT(f_\circ)={\bf o}\hCT{\bf d}(f_\circ)={\bf o}\hCT((df)_*)\geq{\bf o}(\CT df)_*\geq(o\CT df)_\circ$$ and 
 $$\hT(f^\circ)={\bf o}\hCT{\bf d}(f^\circ)={\bf o}\hCT((df)^*)\geq{\bf o}(\CT df)^*\geq(o\CT df)^\circ,$$ 
 for every map $f:X\to Y$ over $\CQ_0$.
 Hence $\hT$ is a lax extension of the functor $T$.

To check that $\hT$ satisfies condition (\ref{mon_ext_Rel}), we first notice that $o$ and {\bf o} respect whiskering, in the sense that, for all $\CQ$-functors $f:W\to X,g:Z\to Y$ and $\CQ$-distributors $\phi:X\oto Y$, one has
\[(og)^\circ\circ({\bf o}\phi)\circ(of)_\circ=
(og)^\circ\circ 1^*_Y\circ({\bf o}\phi)\circ 1^*_X\circ(of)_\circ
={\bf o} g^*\circ{\bf o}\phi\circ{\bf o} f_*={\bf o}(g^*\circ\phi\circ f_*).\]
Thus, for every $\CQ$-relation $r:X\rto Y$ one obtains,
\[r={\bf od} r\leq{\bf o}(\se_{dY}^*\circ\hCT{\bf d} r\circ(\se_{dX})_*)
=(o\se_{dY})^\circ\circ{\bf o}\hCT{\bf d} r\circ(o\se_{dX})_\circ=e_X^\circ\circ\hT r\circ (e_Y)^\circ,
\]
and
\begin{align*}
(m_Y)^\circ\circ\hT r\circ(m_X)_\circ
&=(o(\sm_{dY}\cdot\CT\ep_{\CT dY}))^\circ\circ{\bf o}\hCT{\bf d} r\circ(o(\sm_{dX}\cdot\CT\ep_{\CT dX}))_\circ\\
&={\bf o}((\sm_{dY}\cdot\CT\ep_{\CT dY})^*\circ\hCT{\bf d} r\circ(\sm_{dX}\cdot\CT\ep_{\CT dX})_*)\\
&={\bf o}((\CT\ep_{\CT dY})^*\circ\sm_{dY}^*\circ\hCT{\bf d} r\circ(\sm_{dX})_*\circ(\CT\ep_{\CT dX})_*)\\
&\geq{\bf o}(\CT\ep_{\CT dY})^*\circ\hCT\hCT{\bf d} r\circ(\CT\ep_{\CT dX})_*)\\
&={\bf o}\hCT(\ep_{\CT dY}^*\circ\hCT{\bf d} r\circ(\ep_{\CT dX})_*)\\
&={\bf o}\hCT{\bf d}{\bf o}\hCT{\bf d} r=\hT\hT r;
\end{align*}
here the penultimate equality holds since $\ep_Y^*\circ\varphi\circ(\ep_X)_*={\bf do}\varphi$ for every $\CQ$-distributor $\varphi: X\oto Y$ \end{proof}

To extend the process $\Gamma:(\CT,\hat{\CT})\mapsto(o\CT d,{\bf o}\hat{\CT}{\bf d})$ functorially,
let $\kappa:(\frak{S},\hat{\CS})\to(\frak{T},\hCT)$ be a morphism of lax extensions of $2$-monads on $\QCat$. Then $\lam_X=o\kappa_{dX}$ defines a morphism 
$\Gamma(\kappa)=\lam:(o\CS d,{\bf o}\hat{\CS}{\bf d})\to(o\CT d,{\bf o}\hCT{\bf d})$ between lax extensions of monads on $\QSet$, and the remaining verifications of the assertion of the following theorem become a straightforward exercise.

\begin{thm}\label{inverse} $\Gamma: \mathsf{ExtMnd}(\QCat)\to\mathsf{ExtMnd}(\QSet)$ is a functor that is left inverse to the embedding $\Delta$.
\end{thm}
In the next section we will restrict the domain of $\Gamma$, so that the respective restrictions of $\Gamma$ and $\Delta$ become adjoint.

\section{The minimal lax extension of a $2$-monad on $\QCat$}

At first, let us consider just a $2$-functor $\CT$ on $\QCat$. By Proposition \ref{hT_graph}, any lax extension $\hCT$ of $\CT$ must satisfy
\[\hCT\phi=\hCT((\olphi)^*\circ(\sy_X)_*)=(\CT\olphi)^*\circ\hCT 1^*_{\CP X }\circ(\CT\sy_X)_*\geq(\CT\olphi)^*\circ(\CT\sy_X)_*,\] for all $\CQ$-distributors $\phi:X\to Y$.
Putting
 \[\overline{\CT}\varphi:=(\CT\olphi)^*\circ(\CT\sy_X)_*\]
 for all $\phi$, conversely one proves easily  (see {\rm \cite{Akhvlediani2010,Stubbe2010}}):

\begin{prop}\label{Barr ext}
Every $2$-functor $\CT:\QCat\to\QCat$ admits a least lax extension to $\QDist$,  given by $\overline{\CT}$.
\end{prop}

\begin{cor}\label{uniqueness}
Any flat lax extension to $\QDist$ of an endofunctor $\CT$ must equal $\ovl{\CT}$.
\end{cor}

But is $\ovl{\CT}$ always flat? By definition one has
\[\overline{\CT}1^*_X=(\CT(\ola{1_X^*}))^*\circ(\CT\sy_X)_*=(\CT\sy_X)^*\circ(\CT\sy_X)_*,\]
so that the right-hand-side term needs to equal $1_{\CT X}^*$ in order for us to guarantee that $\ovl{\CT}$ be flat; that is:
like $\sy_X$, the $\CQ$-functor $\CT\sy_X$ needs to be fully faithful. This and the previous arguments easily prove the implications ${\rm (i)\Lra(ii)\Lra(iii)\Lra(iv)\Lra(i)}$ of the following Proposition.

\begin{prop}\label{uni_fla_ext} For a $2$-functor $\CT:\QCat\lra\QCat$, the following assertions are equivalent:
\begin{enumerate}[(i)]%[label={\rm (\arabic*)}]
\item $\CT\sy_X$ is a fully faithful $\CQ$-functor, for every $\CQ$-category $X$;
\item $\overline{\CT}$ is a flat lax extension of $\CT$ to $\QDist$;
\item $\CT$ admits some flat lax extension to $\QDist$;
\item $\CT$ admits a unique flat lax extension to $\QDist$.
\end{enumerate}
\end{prop}

Under these equivalent conditions we are able to prove that, when $\CT$ carries a monad structure $\frak{T}=(\CT,\sm,\se)$, the minimal lax extension $\ovl{\CT}$ of the 2-functor $\CT$ actually constitutes a lax extension of the 2-monad $\frak{T}$
to $\QDist$:

\begin{thm}\label{fla_imp_mon}
Let $\frak{T}=(\CT,\mathsf{m},\mathsf{e})$ be a $2$-monad on $\QCat$ such that $\CT$ preserves the full fidelity of all Yoneda functors. Then $\ovl{\CT}$ is a flat lax extension of the monad $\frak{T}$ to $\QDist$, and it is the only one.
\end{thm}
\begin{proof}
Given a distributor $\varphi:X\oto Y$, one has
\begin{align*}
(\mathsf{e}_Y)^*\circ\overline{\CT}\varphi\circ(\mathsf{e}_X)_*
&=(\mathsf{e}_Y)^*\circ(\CT\olphi)^*\circ(\CT\sy_X)_*\circ(\mathsf{e}_X)_*\\
&=(\CT\olphi\cdot\mathsf{e}_Y)^*\circ(\CT\sy_X\cdot\mathsf{e}_X)_*\\
&=(\mathsf{e}_{\sP X}\cdot\olphi)^*\circ(\mathsf{e}_{\sP X}\cdot\sy_X)_*\\
&=(\olphi)^*\circ(\mathsf{e}_{\sP X})^*\cdot(\mathsf{e}_{\sP X})_*\circ(\sy_X)_*\\
&\geq (\olphi)^*\circ(\sy_X)_*\\
&=\phi,
\end{align*}
and
\begin{align*}
(\mathsf{m}_Y)^*\circ\overline{\CT}\phi\circ(\mathsf{m}_X)_*
&=(\mathsf{m}_Y)^*\circ(\CT\olphi)^*\circ(\CT\sy_X)_*\circ(\mathsf{m}_X)_*\\
&=(\CT\olphi\cdot\mathsf{m}_Y)^*\circ(\CT\sy_X\cdot\mathsf{m}_X)_*\\
&=(\mathsf{m}_{\sP X}\cdot\CT\CT\olphi)^*\circ(\mathsf{m}_{\sP X}\cdot\CT\CT\sy_X)_*\\
&=(\CT\CT\olphi)^*\circ(\mathsf{m}_{\sP X})^*\circ(\mathsf{m}_{\sP X})_*\circ(\CT\CT\sy_X)_*\\
&\geq(\CT\CT\olphi)^*\circ(\CT\CT\sy_X)_*\\
&=\overline{\CT}((\CT\olphi)^*\circ(\CT\sy_X)_*)\\
&=\overline{\CT}\,\overline{\CT}\phi,
\end{align*}
with the penultimate equality holding since $\ovl{\CT}$ is flat, by Proposition \ref{uni_fla_ext}. Therefore, $\ovl{\CT}$ is a lax extension of $(\CT,\sm,\se)$ by (\ref{ext_of_2-mon}).
\end{proof}

The corresponding assertion of the Theorem is not valid in the discrete environment: for a monad on $\QSet$, there may exist several flat extensions to $\QRel$, as the following easy example shows.
\begin{exmp} Let $\CQ$ be the Lawvere quantale $([0,\infty],\geq,+,0)$, Then, trivially,  the identity functor on $\QRel$ provides a flat lax extension of the identity monad $\bbI$ on $\Set$ to $\QRel$, but so does $\check{I}$ defined by
%First, Let $\hat{I}$ be the identity functor on $\QRel$, then it is clearly a flat lax extension of $\bbI$.
%Second, for a $\CQ$-relation $r:X\rto Y$, let 
\[\check{I}r(x,y)=\begin{cases}
0 & \text{if}\ r(x,y)<\infty,\\
\infty & \text{if}\ r(x,y)=\infty, 
\end{cases}\]
for every $\CQ$-relation $r:X\rto Y.$
\end{exmp}

%\begin{prop}
%For any lax extension $(\bbT,\hT)$ of the monad $\bbT$ on $\QSet$, the lax extension $\Delta(\bbT,\hT)$ is a flat lax extension of the monad $\Delta\bbT$ on $\QCat$.
%\end{prop}

We already proved in Proposition \ref{ext_ext} that, for every lax extension $\hT$ of a monad $\bbT$ on $\QSet$, $\Delta(\bbT,\hT)$ gives a $2$-monad on $\QCat$ with a {\em flat} lax extension to $\QDist$. Hence, when we denote by 
$\mathsf{FlatExtMnd}(\QCat)$
the full subcategory of $\mathsf{ExtMnd}(\QCat)$ of flat lax extensions of monads on $\QCat$, we know that the functor
$\Delta$ of Theorem \ref{Delta} takes values in that subcategory. We will now show that the respective restrictions
 \[\mathsf{ExtMnd}(\QSet)\to\mathsf{FlatExtMnd}(\QCat)\to\mathsf{ExtMnd}(\QSet)\] 
 of $\Delta$ and $\Gamma$ are adjoint to each other. In order to simplify matters, let us replace $\mathsf{FlatExtMnd}(\QCat)$ by the isomorphic full subcategory $\mathsf{ffYMnd}\QCat$ of the metacategory $\mathsf{Mnd}(\QCat)$, containing those $2$-monads of $\QCat$ whose endofunctor preserves the {\em full fidelity of all Yoneda $\CQ$-functors}: see
 Proposition \ref{uni_fla_ext}(i) and Theorem \ref{fla_imp_mon}. Hence, we will show that the functors
 \[\mathsf{ExtMnd}(\QSet)\to^{\Delta'}\mathsf{ffYMnd}(\QCat)\to^{\Gamma'}\mathsf{ExtMnd}(\QSet)\] 
 are adjoint to each other: $\Delta'\dashv\Gamma'$. Here $\Delta'$ assigns to $(\bbT,\hT)$ the monad defined by (\ref{ext_fun}) and (\ref{ext_mon}), and $\Gamma'$ assigns to 
 $\mathfrak{T}=(\CT,\sm,\se)$
  the monad defined by (\ref{res_mon}), provided with the lax extension $\ovl{T}={\bf o}\ovl{\CT}{\bf d}$ (see Proposition \ref{non_dis_to_dis}).
 
 \begin{thm}\label{coref} $\Delta'$ embeds $\mathsf{ExtMnd}(\QSet)$ into $\mathsf{ffYMnd}(\QCat)$ as a full coreflective subcategory.
\end{thm}

\begin{proof}  Since, by Theorem \ref{inverse}, $\Gamma'$ is left inverse of $\Delta'$, the identity transformations $1_T$
are the obvious candidates to serve as the units 
$(\bbT,\hT)\to \Gamma'\Delta'(\bbT,\hT)$ 
of the adjunction, where $\mathbb{T}=(T,m,e)$. We will now describe the counits 
\[\iota^{\CT}:\Delta'\Gamma'\mathfrak{T}\to\mathfrak{T},\]
for all $2$-monads $\mathfrak{T}=(\CT,\sm,\se)$ on $\QCat$, such that $\CT$ makes the $\CQ$-functors $\CT\sy_X$ fully faithful, for all $\CQ$-categories $X=(X,a)$.
To this end, with the counit $\ep:do\to1$ of the adjunction $d\dashv o$ described in Section 2, consider the $\CQ$-functors $\CT\ep_{(X,a)}:\CT(X,1_X^{\circ})=\CT do(X,a)\to\CT(X,a)$. 
Since $\Gamma'$ assigns to $\CT$ the monad $\bbT$ with its endofunctor $T=o\CT d$ laxly extended by $\ovl{T}={\bf o}\ovl{\CT}{\bf d}$, we can define the $(X,a)$-component 
\[\iota^{\CT}_{(X,a)}:(\Delta'\Gamma'\CT)(X,a)=(TX,\ovl{T}a)\to\CT(X,a)\]
of $\iota^{\CT}$ to have the same underlying maps as $\CT\ep_{(X,a)}$, so that $o\,\iota^{\CT}_{(X,a)}=o\CT\ep_{(X,a)}$, but must first verify that $\iota^{\CT}_{(X,a)}$ is indeed a $\CQ$-functor. But since the $\CQ$-relation $a:X\rto X$ 
satisfies $a=a\circ a$, so that ${\bf d}a=(\ep_{(X,a})^*\circ(\ep_{(X,a)})_*$, and since $o$ and ${\bf o}$ respect whiskering (see the proof of Proposition \ref{non_dis_to_dis}) and $\ovl{\CT}$ is flat, we see that
\begin{align*}
\ovl{T} a&={\bf o}\ovl{\CT}{\bf d}\, a\\
&={\bf o}\ovl{\CT}((\ep_{(X,a)})^*\circ (\ep_{(X,a)})_*)\\
&={\bf o}((\CT \ep_{(X,a)})^*\circ(\CT \ep_{(X,a}))_*)\\
&=(o\CT\ep_{(X,a)})^\circ\circ1_{\CT(X,a)}^*\circ(o\CT\ep_{(X,a)})_\circ\\
&=(o\, \iota^{\CT}_{(X,a)})^{\circ}\circ 1_{\CT(X,a)}^*\circ(o\, \iota^{\CT}_{(X,a)})_{\circ}.
\end{align*}
Naturality of $\iota^{\CT}$ follows immediately from the naturality of $\CT\ep$ since the corresponding naturality diagrams have the same underlying maps, and it is also easy to see that $\iota^{\CT}$ respects the monad structures. Indeed, since the left square of
\[
\begin{array}{cc}
\bfig \square<800,500>[do`1_{\QCat}`\CT do`\CT;\ep`edo`e`\CT\ep] \efig &
\bfig \square<800,500>[\Delta'\Gamma'1_{\QCat}`1_{\QCat}`\Delta'\Gamma'\CT`\CT;1`\Delta'\Gamma'e`e`\iota^{\CT}] \efig
\end{array}
\]
commutes, so does the right one, since it has the same underlying maps as the left one when read componentwise for every $(X.a)$ in $\QCat$. Consequently, $\iota^{\CT}$ respects the units of the monads, and for the preservation of the multiplication one argues similarly. 

To check that $\iota^{\CT}$ is natural in $\CT$, we just need to observe that, for a morphism $\alpha:\CT\to\CS$ of $2$-monads, the left square of 
\[
\begin{array}{cc}
\bfig \square<800,500>[\CT do`\CS do`\CT`\CS;\alpha do`\CT\ep`\CS\ep`\alpha] \efig &
\bfig \square<800,500>[\Delta'\Gamma'\CT`\Delta'\Gamma'\CS`\CT`\CS;\Delta'\Gamma'\alpha`\iota^{\CT}`\iota^{\CS}`\alpha] \efig
\end{array}
\]
commutes; consequently, since the right diagram has the same underlying maps as the left one when read componentwise for every $\CQ$-category $(X,a)$, it commutes as well.

Finally, for the verification of the triangular identities of the desired adjunction, since the units are identity morphisms, it suffices to show that all morphisms $\Gamma'\iota^{\CT}:\Gamma'\CT\to \Gamma'\CT$ and 
$\iota^{\Delta'(\bbT,\hT)}:\Delta'(\bbT,\hT)\to\Delta'(\bbT,\hT)$ are identity transformations. But this follows again from the fact that the components of $\ep d$ and $o\,\ep$ are identity morphisms, since with $\CT'=\Delta'(\bbT,\hT)$ one has
\[o((\Gamma'\iota^{\CT})_X)=o\,\iota^{\CT}_{dX}=o\CT\ep_{dX}\quad\text{and}\quad o\,\iota{\CT'}_{(X,a)}=o\CT'\ep_{(X,a)} =To\ep_{(X,a)},\]
respectively for all sets $X$ over $\CQ$ and all $\CQ$-categories $(X,a)$.
% obviously the unit. Second, the transformation $\mu$ defined in the proof of Proposition \ref{ini_of_fla} is clearly a morphism from $\Delta'\Sigma'(\frak{T},\ovl{\CT})$ to $(\frak{T},\ovl{\CT})$. For a lax extension $(\bbS,\hat{S})$ and a morphism $\lam:\Delta'(\bbS,\hat{S})=(\frak{S},\ovl{\CS})\to(\frak{T},\ovl{\CT})$,
%let \[\tilde{\lam}=\Sigma\lam:\Sigma\Delta(\bbS,\hat{S})
%=(\bbS,\hat{S})\to\Sigma(\frak{T},\ovl{\CT}),\]
%define a transformation $\tilde{\lam}:S\to T=G\CT D$ by $\tilde{\lam}_X=G\lam_{DX}:SX\to TX=G\CT DX$. It is easily seen that $\tilde{\lam}$ is a morphism from the monad $\bbS$ to $\bbT=(G\CT D,m,e)$ and also a morphism from $(\bbS,\hat{S})$ to $(\bbT,\hat{T})=\Sigma'(\frak{T},\ovl{\CT})$.It is routine to verify that $\mu\circ\Delta'\tilde{\lam}=\lam$, so that $\mu$ is universal
%since the lax extension $\Delta'\Sigma'(\frak{T},\ovl{\CT})$ is the initial one induced by $\mu$.
\end{proof}

\begin{rem}\label{ini_of_fla} (1) For a flat lax extension $(\frak{T},\ovl{\CT})$ of a $2$-monad $\frak{T}=(\CT,\sm,\se)$ on $\QCat$, with the calculation given in the proof above one sees that $\iota^{\CT}:\Delta\Gamma(\frak{T},\ovl{\CT})\to(\frak{T},\ovl{\CT})$ is an initial morphism with respect to the topological functor of Proposition \ref{top_cat}.

(2) According to Theorem \ref{coref}, $\mathsf{ExtMnd}(\QSet)$ is eqivalent to the full coreflective subcategory of $\mathsf{ffYMnd}(\QCat)$ formed by
all those $2$-monads $\mathfrak{T}=(\CT.\sm,\se)$ on $\QCat$ for which $\iota^{\CT}$ is an isomorphism. Up to monad isomorphism, they may be described by the following characteristic property of their endofunctor $\CT$:  for all 
$\CQ$-categories $(X,a)$, the underlying set over $\CQ_0$ of $\CT(X,a)$ is equal to that of $\CT(X,1_X^{\circ})$,
 and the structure $1^*_{\CT(X,a)}$ may be computed as ${\bf o}\ovl{\CT}{\bf d}a$, which amounts to
\[1^*_{\CT(X,a)}(\Fx,\Fy)=1^*_{\CT\CP(X,1_X^{\circ})}((\CT \sy_{(X,1_X^{\circ})})(\Fx),(\CT\ola{a})(\Fy)) \]
for all elements $\Fx,\Fy$ in the underlying set of $\CT(X,a)$; here $\ola{a}$ is to be treated as a $\CQ$-functor $(X,1_X^{\circ})\to \CP(X,1_X^{\circ})$. Roughly, $\CT$ must be completely determined by its effect on discrete $\CQ$-categories and discrete presheaf $\CQ$-categories, and their $\CQ$-functors.

(3) Without reference to Theorem \ref{coref}, already from Proposition \ref{uni_fla_ext} and Theorem \ref{inverse} we may infer that, for a given monad $\bbT=(T,m,e)$ on $\Set/\CQ_0$, the assignment
\[(\bbT,\hT)\mapsto(\CT,\sm,\se)\] 
defined by equations {\rm (\ref{ext_fun})} and {\rm (\ref{ext_mon})} establishes a bijection between the lax extensions of $\bbT$ and the $2$-monads $\frak{T}=(\CT,\mathsf{m},\mathsf{e})$ on $\QCat$ such that $o\CT=To$, $o\sm=mo$, $o\se=eo$, and $\CT$ preserves the full fidelity of all Yoneda functors.
\end{rem}

\section{The presheaf and Hausdorff monads on $\QCat$ and their discrete coreflections}

We now present some old and new examples of $2$-monads on $\QCat$, starting with the presheaf and copresheaf monads and their composites, as investigated in \cite{LST}, from which we then obtain new ``Hausdorff-type" submonads, one of which was introduced in \cite{Akhvlediani2010} and further developed in \cite{Stubbe2010}. An application of the coreflector $\Sigma'$ to them leads to various known lax extensions of $\Set$-monads, as presented in \cite{Hofmann2014} and previous works, such as \cite{Clementino2003,Clementino2004}.
 
  %are listed as applications of the functor $\Sigma$. %When the quantaloid is a quantale $\CV$ and the boolean algebra $\mathsf{2}=(\{0,1\},\wedge)$, some important lax extensions in \cite{Hofmann2014} are induced.

For the presheaf $2$-monad $(\CP,\sfs,\sy)$, the minimal lax extension of $\CP$ to $\QDist$ may be computed as follows: 
\begin{align*}
\overline{\CP}\phi(\si,\tau)
&=(((\overleftarrow{\phi})_!)^*\circ((\sy_X)_!)_*)(\si,\tau)\\
&=1^*_{\CP\CP X}(\si\circ(\sy_X)^*,\tau\circ(\overleftarrow{\phi})^*)\\
&=1^*_{\CP X}(\si,\tau\circ(\overleftarrow{\phi})^*\circ(\sy_X)_*)\\
&=1^*_{\CP X}(\si,\tau\circ\phi)
\end{align*}
for every $\CQ$-distributor $\phi:X\oto Y$ and all $\si\in\CP X,\tau\in\CP Y$; consequently, 
\begin{equation}\label{ext_presheaf}
\overline{\CP}\phi=(\phi^\odot)^*.
\end{equation}
%Therefore, $$\overline{\CP}\rho=(\rho^\odot)^*.$$
%Moreover, $\CP\sy_X=(\sy_X)_!$ is fully faithful. In fact, it holds that for all $\phi,\psi$ in $\CP X$,
%Clearly, for each $\CQ$-category $X$,
%\[\overline{\CP}1^*_X=1^*_{\CP X},\] therefore, $\overline{\CP}$ is a flat lax extension.
%\[\CP\CP X((\sy_X)_!\phi,(\sy_X)_!\psi)=\CP\CP X(\phi\circ(\sy_X)^*,\psi\circ(\sy_X)^*)=\CP X(\phi,\psi\circ(\sy_X)^*\circ(\sy_X)_*)\CP X(\phi,\psi).\] Hence,
Similarly one obtains
\begin{equation}\label{ext_copresheaf}
\overline{\CP^\dagger}\phi=(\phi^\oplus)_*,
\end{equation} for all distributors $\phi:X\oto Y$.
% and $$\overline{\CP}^\dagger 1^*_X=1^*_{\CP^\dagger X}$$ for all $\CQ$-category $X$. Thus, $\overline{\CP}^\dagger$ is a flat extension of $\CP^\dagger$.

The two composites of $\CP$ and $\CP^\dagger$ give rise to two $2$-monads on $\QCat$ \cite{Stubbe2013, Pu2015, LST}. One is the double presheaf $2$-monad
$(\CP\CP^\dagger,\Fs,\Fy)$, whose multiplication $\Fs$ and unit $\Fy$ are given by
\[\Fs_X=(\syd_{\CP\CPd X}\cdot \sy_{\CPd X})_*^\odot:\CP\CP^\dagger\CP\CP^\dagger X\to\CP\CP^\dagger X,\]
\[\Fy_X=\sy_{\CP^\dagger X}\cdot\sy^\dagger_X:X\to\CP\CP^\dagger X.\]
The other is the double copresheaf $2$-monad $(\CP^\dagger\CP,\Fs^\dagger,\Fy^\dagger)$, whose multiplication $\Fsd$ and unit $\Fyd$ are given by
\[\Fsd_X=((\sy_{\CPd\CP X}\cdot\syd_{\CP X})^*)^{\oplus}:\CPd\CP\CPd\CP X\to\CPd\CP X,\]
\[\Fyd_X=\syd_{\CP X}\cdot\sy_X:X\to\CPd\CP X.\]
The minimal lax extensions of $\CP\CP^{\dagger}$ and $\CP^{\dagger}\CP$ may be expressed in terms of the lax extensions $\ovl{\CP},\ovl{\CP^\dagger}$, as follows: 
\begin{equation}\label{double}
\overline{\CP\CPd}(\phi)=(\CP\CPd\olphi)^*\circ(\CP\CPd\sy_X)_*
=\overline{\CP}((\CPd\olphi)^*\circ(\CPd\sy_X)_*)
=\overline{\CP}\,\overline{\CP^\dagger}\phi
\end{equation}
and
\begin{equation}\label{codouble}
\overline{\CPd\CP}(\phi)=(\CPd\CP\olphi)^*\circ(\CPd\CP\sy_X)_*
=\overline{\CP^\dagger}((\CP\olphi)^*\circ(\CP\sy_X)_*)
=\overline{\CP^\dagger}\overline{\CP}\phi,
\end{equation}
for all $\CQ$-distributors $\phi:X\oto Y$. 
%Since both $\overline{\CP}$ and $\overline{\CP}^\dagger$ are flat,  $\overline{\CP\CPd}$ and $\overline{\CPd\CP}$ are flat lax extension.

\begin{prop}\label{flat_presheaf}
The minimal lax extensions of the (co)presheaf functors $\CP,\,\CP^{\dagger}$ and the double presheaf functors $\CP\CP^{\dagger},\,\CP^{\dagger}\CP$ are flat and therefore provide the unique flat extensions to $\QDist$ of the corresponding monads on $\QCat$.
\end{prop}

\begin{proof}
From (\ref{ext_presheaf}) one obtains $\overline{\CP}1^*_X=1^*_{\CP X}$, so that $\ovl{\CP}$ is flat. Likewise, the flatness of $\CP^{\dagger}$ follows from (\ref{ext_copresheaf}), and $\CP\CP^{\dagger}$ and $\CP^{\dagger}\CP$ are flat since $\CP$ and $\CP^{\dagger}$ are, according to (\ref{double}) and (\ref{codouble}).
\end{proof}

\begin{exmp}\label{firstexmp}
We describe the coreflections into $\mathsf{ExtMnd}(\QSet)$ of the monads of Proposition \ref{flat_presheaf} when $\CQ=\mathsf{2}$ is the two-element chain. Both $o\CP d$ and $o\CPd d$ reduce to the usual powerset functor $P$  on $\Set$ and are provided with the usual monad structure. The lax extensions ${\bf o}\overline{\CP} {\bf d}$ and 
${\bf o}\overline{\CP}^\dagger {\bf d}$ give respectively the lax extension $\check{P}$ and $\hat{P}$, described in \cite{Hofmann2014} as
\[A(\check{P}r)B\iff\forall x\in A\exists y\in B\,(x\,r\,y)\quad\text{and}\quad A(\hat{P}r)B\iff \forall y\in B\exists x\in A\,(x\,r\,y),\]
for all relations $r:X\rto Y\,\text{and}\, A\subseteq X, B\subseteq Y$.

Both $o\CP\CPd d$ and $o\CPd\CP d$ reduce to the usual up-set functor 
\[U:\Set\to\Set, \quad UX=\{\frak{a}\subseteq PX\,|\,\forall A\in\frak{a}, A\subseteq B\Lra B\in\frak{a}\}.\] 
In fact, the monads $(\CP\CPd,\Fs,\Fy)$ and $(\CPd\CP,\Fsd,\Fyd)$ have the up-set monad $\mathbb{U}=(U,m,e)$ (see  \cite{Hofmann2014}) as their coreflection into $\QSet$, but
${\bf o}\overline{\CP\CPd}{\bf d}$ and ${\bf o}\overline{\CPd\CP}{\bf d}$ reduce to distinct lax extensions $\check{U}$ and $\hat{U}$ of $U$, given respectively by
\[\frak{a}(\check{U}r)\frak{b}\iff \forall A\in\frak{a}\,\exists B\in\frak{b}\,\forall y\in B\,\exists x\in A\,( x\ r\ y),\]
\[\frak{a}(\hat{U}r)\frak{b}\iff \forall B\in\frak{b}\,\exists A\in\frak{a}\,\forall x\in A\,\exists y\in B\,( x\ r\ y),\]
for all relations $r:X\rto Y$ and $\frak{a}\in UX, \frak{b}\in UY$.

The filter monad $\mathbb{F}$ and the ultrafilter monad are submonads of $\mathbb{U}$, and one may obtain the lax extensions $\hat{\mathbb{F}}$, $\check{\mathbb{F}}$, as well as the Barr extension of the ultrafilter monad, by initial extensions, as described in \cite{Hofmann2014}.
\end{exmp}

A presheaf $\si:(X,a)\oto \{|\si|\}$ (copresheaf $\tau:\{|\tau|\}\oto(X,a)$) is {\em conical} if there is some $A\subseteq X$ with $|x|=|\si|$ ($|x|=|\tau|$) for all $x\in A$, such that $\si=\bv_{x\in A}a(-,x)$ ($\tau=\bv_{x\in A}a(x,-)$, respectively). The conical presheaves (copresheaves) of a $\CQ$-category $X$ equipped with the structure inherited from $\CP X$ ($\CPd X$) form a $\CQ$-category, denoted by $\CH X$ ($\CH^\dagger X$, respectively). 

Clearly, $\CH X$ ($\CHd X$) is closed under the formation of joins in $\QDist(X,\{s\})\,(\QDist(\{s\},X))$ of families of (co)presheaves on $X$ of fixed type $s$:

\begin{lem} Let $X$ be a $\CQ$-category, and let
$\{\si_i:i\in I\}$ be a family of (co)presheaves in $\CH X$ ($\CHd X$), with all $\si_i$ having the same type. 
Then also  $\bv_{i\in I}\si_i$ lies in $\CH X$ ($\CHd X$, respectively).
\end{lem}

For a $\CQ$-functor $f:X\to Y$ and every conical presheaf $\si$ of $X$,  $\si\circ f^*$ is a conical presheaf of $Y$. Hence, by restricting the domain and codomain of $(f^*)^\odot:\CP X\to\CP Y$ to $\CH X$ and $\CH Y$ respectively, one obtains a $\CQ$-functor, again denoted by $(f^*)^\odot:\CH X\to \CH Y$, and $\CH$ becomes a $2$-functor $\CH:\QCat\to\QCat$. In fact, $\CH$ belongs to a sub-$2$-monad of $(\CP,\sm,\se)$ on $\QCat$ with the unit given by the restrictions $\sy_X:X\to\CH X$ of the Yoneda functors and the multiplication by
\[\sfs_X:\CH\CH X\to\CH X,\quad \Sigma \mapsto\Sigma\circ(\sy_X)_*=\bv_{i\in I}\si_i,\] 
for all $\Sigma\in\CH\CH X$
 with $\Sigma=\bv_{i\in I}1^*_{\CH X}(-,\si_i)$. Following the work of \cite{Akhvlediani2010} when $\CQ$ is a quantale, the $2$-monad $(\CH,\sh,\sy)$ was called the {\em Hausdorff doctrine} in \cite{Stubbe2010} also for an arbitrary small quantaloid.

Dually, $X\mapsto\CHd X$ advances to a sub-$2$-monad of $(\CPd,\ssd,\syd)$, called the {\em co-Hausdorff doctrine}, which maps a $\CQ$-functor $f:X\to Y$ to $(f_*)^\oplus:\CHd X\to\CHd Y$, and is equipped with the unit given by the co-Yoneda functors
$\syd_X:X\to\CH^\dagger X$ and the multiplication given by
\[\sfs^\dagger_X:\CH^\dagger\CH^\dagger X\to\CH^\dagger X,\quad \Theta\mapsto(\syd_X)^*\circ\Theta=\bv_{i\in I}\tau_i,\] for all $\Theta\in\CH^\dagger\CH^\dagger X$ with $\Theta=\bv_{i\in I}1^*_{\CH^\dagger X}(\tau_i,-)$.
The minimal lax extension $\ovl{\CH}$ appears in \cite{Akhvlediani2010, Stubbe2010}:

\begin{prop} {\rm (1)} The Hausdorff doctrine $(\CH,\sh,\sy)$ on $\QCat$ admits the flat lax extension $\overline{\CH}$ to $\QDist$ given by
 \[\overline{\CH}\phi(\alpha,\beta)=\bw_{x\in A}\bv_{y\in B}\phi(x,y),\] 
 for all $\phi:(X,a)\oto(Y,b)$ and $\alpha=\bv_{x\in A}a(-,x)\in\CH X$,  $\beta=\bv_{y\in B}b(-,y)\in\CH Y$.

{\rm (2)} The co-Hausdorff doctrine $(\CHd,\ssd,\syd)$ on $\QCat$ admits the flat lax extension $\overline{\CH^\dagger}$ to $\QDist$ given by
\[\overline{\CH}^\dagger\phi(\alpha,\beta)=\bw_{y\in B}\bv_{x\in A}\phi(x,y),\] 
for all $\phi:(X,a)\oto(Y,b)$ and $\alpha=\bv_{x\in A}a(x,-)\in\CH^{\dagger}X,\, \beta=\bv_{y\in B}b(y,-)\in\CH^{\dagger}Y$.
 \end{prop}

\begin{exmp}
When the quantaloid is a unital quantale $\CV$, both $o\CH d$ and $o\CH^\dagger d$ give the ordinary powerset functor on $\Set$, with its usual monad structure, and ${\bf o}\overline{\CH}{\bf d}$ gives its Kleisli extension $\check{P}$ to $\VRel$, described in \cite{Hofmann2014} (Exercise IV.2.H) by
\[\check{P}r(A,B)=\bw_{x\in A}\bv_{y\in B}r(x,y),\] for all relations $r:X\rto Y$ and all $A\subseteq X,\,B\subseteq Y$, while 
${\bf o}\overline{\CH^\dagger}{\bf d}$ gives another lax extension $\hat{P}$, described by
\[\hat{P}r(A,B)=\bw_{y\in B}\bv_{x\in A}r(x,y).\] 
Briefly: $({\mathbb P},\check{P})$ is the coreflection of $\CH_\CV=\CH$ into $\mathsf{ExtMnd}(\Set)$, and likewise in the dual case. Conversely, embedding $({\mathbb P},\check{P})$ into 
$\mathsf{ffYMnd}(\CV\text{-}\Cat)$ via Theorem \ref{coref} gives us back the monad $\CH_\CV$ as introduced in \cite{Akhvlediani2010}, and likewise in the dual case.
\end{exmp}

%\begin{exmp}
%The lax extension $(\check{P},\cup,\{\cdot\})$ of the powerset monad $(P,\cup,\{\cdot\})$ to $\VRel$ induces a monad $\CH_\CV:\VCat\to\VCat$ by $\CH_\CV(X,a)=(PX,\check{P}a)$, called Hausdorff monad in \cite{Akhvlediani2010}. The extension $\overline{\CH}_\CV$ given by Proposition \ref{dis_ext_to_non} coincides with that in \cite{Akhvlediani2010}.
%\end{exmp}

%Let $\CQ$ be a quantaloid with each hom lattice $\CQ(p,q)$ being completely distributive. We claim that both $\CH\CH^\dagger:\VCat\to\VCat$ and $\CH^\dagger\CH:\VCat\to\VCat$ are $2$-monads on $\VCat$.

%Since both $\CH$ and $\CHd$ are sub $2$-functors of $\CP$ and $\CPd$ respectively, \[\CH\CHd\hookrightarrow\CH\CPd\hookrightarrow\CP\CPd\] means that $\CH\CHd$ is a sub $2$-functor of $\CP\CPd$. First, restricting the codomain of the unit $\Fy:1_\QCat\to\CP\CPd$ to $\CH\CHd$, one obtains a natural transform , also denoted by $\Fy:1_\Set\to\CH\CHd$. Second, we show that the domain and codomain of  counit $\Fs:\CP\CPd\CP\CPd\to\CP\CPd$ can be restricted to $\CH\CHd\CH\CHd$ and $\CH\CHd$ respectively.
 
In order to establish the existence of sub-2-monads $\CH\CHd,\CHd\CH$ of $\CP\CPd, \CPd\CP$, respectively, we assume that the quantaloid $\CQ$ be {\em completely distributive}, that is: all hom lattices $\CQ(p,q)$ of $\CQ$ need to be completely distributive. This assumption will enable us to show the needed closure property for conical presheaves, as follows.

\begin{lem}\label{meet}
 Let $\CQ$ be completely distributive and $X$ be a $\CQ$-category. Then,
 for every family $\{\Sigma_i:i\in I\}$ in $\CH\CHd X$ with all $\Sigma_i$ having the same type, $\bw_{i\in I}\Sigma_i$ lies also in $\CH\CHd X$.
\end{lem}

\begin{proof} 
Let $|\Sigma_i|=p$ for all $i\in I$. For every $i\in I$, there is a family $(\tau_j:j\in J_i)$ in $\CHd X$ with $|\tau_j|=p$ for all $j\in J_i$, such that $\Sigma_i=\bv_{j\in J_i}1^*_{\CHd X}(-,\tau_j)$. 
Putting $J=\prod\limits_{i\in I}J_i$, with the complete distributivity we obtain
\[\bw_{i\in I}\Sigma_i=\bw_{i\in I}\bv_{j\in J_i}1^*_{\CHd X}(-,\tau_j)
=\bv_{s\in J}\bw_{i\in I}1^*_{\CHd X}(-,\tau_{s(i)})=\bv_{s\in J}1^*_{\CHd X}(-,\bv_{i\in I}\tau_{s(i)}).\] 
Consequently, since each $\bv_{i\in I}\tau_{s(i)}$ is a conical copresheaf of $X$, whence lying in $\CHd X$, $\bw_{i\in I}\Sigma_i$ has been presented as a conical presheaf of $\CHd X$.
\end{proof}
The following Lemma relies on the previous one and will enable us to establish the monad multiplication for $\CH\CHd$.

\begin{lem}\label{monadmult} 
Let $\CQ$ be completely distributive and $X$ be a $\CQ$-category. Then $\Upsilon\circ(\syd_{\CH\CHd X}\cdot\sy_{\CHd X})_*$ is a conical presheaf of $\CHd X$ whenever $\Upsilon$ is a conical presheaf of $\CHd\CH\CHd X$.
\end{lem}
\begin{proof} Let $\Upsilon=\bv_{i\in I}1^*_{\CHd\CH\CHd X}(-,\Xi_i)$, with $\Xi_i=\bv_{j\in J_i}1^*_{\CH\CHd X}(\Sigma_j,-)$ for all $i\in I$. Then
\begin{align*}\Upsilon\circ(\syd_{\CH\CHd X}\cdot\sy_{\CHd X})_*
&=\bv_{i\in I}1^*_{\CHd\CH\CHd X}(\syd_{\CH\CHd X}\cdot\sy_{\CHd X}(-),\Xi_i)\\
&=\bv_{i\in I}1^*_{\CHd\CH\CHd X}(\syd_{\CH\CHd X}\cdot\sy_{\CHd X}(-),\bv_{j\in J_i}1^*_{\CH\CHd X}(\Sigma_j,-))\\
&=\bv_{i\in I}\bw_{j\in J_i}1^*_{\CHd\CH\CHd X}(\syd_{\CH\CHd X}\cdot\sy_{\CHd X}(-),1^*_{\CH\CHd X}(\Sigma_j,-))\\
&=\bv_{i\in I}\bw_{j\in J_i}1^*_{\CH\CHd X}(\sy_{\CHd X}(-),\Sigma_j)\\
&=\bv_{i\in I}\bw_{j\in J_i}\Sigma_j.
\end{align*}
Since. by Lemma \ref{meet} , $\bw_{j\in J_i}\Sigma_j$ is conical for all $i\in I$,
we conclude that $\Upsilon\circ(\syd_{\CH\CHd X}\cdot\sy_{\CHd X})_*$ is conical as well.
\end{proof}

The Lemma gives the crucial tool for establishing the {\em double Hausdorff monad} and the {\em double co-Hausdorff monad} on $\QCat$ and describing their minimal lax extensions to $\QDist$:

\begin{thm}\label{doubleHaus}
When $\CQ$ is completely distributive,
$\CH\CHd$ belongs to a sub-$2$-monad $(\CH\CHd,\Fs,\Fy)$ of $(\CP\CPd,\Fs,\Fy)$ and admits the flat lax extension $\overline{\CH\CHd}$ to $\QDist$ with $\overline{\CH\CHd}=\overline{\CH}\,\overline{\CH^\dagger}$.
Likewise,  there is a $2$-monad $(\CHd\CH,\Fs^\dagger,\Fyd)$ on $\QCat$ admitting the flat lax extension $\overline{\CHd\CH}$ to $\QDist$ with $\overline{\CHd\CH}=\overline{\CH^\dagger}\overline{\CH}$.
\end{thm}

\begin{proof}
Trivially, the units $\Fy_X$ of $(\CP\CPd,\Fs,\Fy)$ take values in $\CH\CHd X$ and may therefore serve as the units for the desired sub-$2$-monad:
\[\Fy_X=\sy_{\CHd X}\cdot\syd_X:X\to\CH\CHd X.\]
That the monad multiplications $\Fs_X$ may be restricted as well to yield $\CQ$-functors
\[\Fs_X=(\syd_{\CH\CHd X}\cdot\sy_{\CHd X})^\odot_*:\CH\CHd\CH\CHd X\to\CH\CHd X\]
is less obvious but follows immediately with Lemma \ref{monadmult}.
It is now clear that $(\CH\CHd,\frak{h},\Fy)$ is a sub-$2$-monad of $(\CP\CPd,\Fs,\Fy)$. That the minimal lax extension $\ovl{\CH\CHd}$ is the composite of the minimal lax extensions of $\CH$ and $\CHd$ follows as in (\ref{double}), which also shows that $\ovl{\CH\CHd}$ is flat.

Dually, one establishes the sub-$2$-monad $(\CHd\CH,\Fs^\dagger,\Fyd)$ of $(\CPd\CP,\Fsd,\Fyd)$, with its unit and multiplication given by 
\[\Fyd_X=\syd_{\CH X}\cdot\sy_X:X\to\CHd\CH X,\]
\[\Fs^\dagger_X=((\sy_{\CHd\CH X}\cdot\syd_{\CH X})^*)^\oplus:\CHd\CH\CHd\CH X\to\CHd\CH X.\]
The assertion about its minimal lax extension follows as in (\ref{codouble}).
\end{proof}

%\begin{thm} The $2$-monad $(\CH\CHd,\frak{h},\Fy)$ admits the flat lax extension $(\overline{\CH\CHd},\frak{h},\Fy)$ to $\QDist$ with $\overline{\CH\CHd}=\overline{\CH}\,\overline{\CH}^\dagger$ and the $2$-monad $(\CHd\CH,\frak{h}^\dagger,\Fyd)$ admits the flat lax extension $(\overline{\CHd\CH},\frak{h}^\dagger,\Fyd)$ with $\overline{\CHd\CH}=\overline{\CH}^\dagger\overline{\CH}$.
%\end{thm}
%begin{proof} Since both $\overline{\CH}$ and $\overline{\CH}^\dagger$ are flat extensions, it follows that
%\[\overline{\CH\CHd}\phi=(\CH\CHd\olphi)^*\circ(\CH\CHd\sy_X)_*
%=\overline{\CH}((\CHd\olphi)^*\circ(\CHd\sy_X)_*)
%=\overline{\CH}\,\overline{\CH}^\dagger\phi\]
%and \[\overline{\CHd\CH}\phi=(\CHd\CHd\olphi)^*\circ(\CHd\CH\sy_X)_*
%=\overline{\CH}^\dagger((\CH\olphi)^*\circ(\CH\sy_X)_*)
%=\overline{\CH}^\dagger\overline{\CH}\phi\] for all distributor $\phi:X\oto Y$.
%\end{proof}

\begin{exmp} When $\CQ$ is a unital and completely distributive quantale $\CV$, both $o\CHd\CH d$ and $o\CH\CHd d$ give the up-set functor $U:\Set\to\Set$ (see Example \ref{firstexmp}); in fact,
both the double Hausdorff and double co-Hausdorff monad induce the up-set monad $\mathbb{U}=(U,m,e)$ on $\Set$. (Note that, of course,  $\CH=\CH_{\CV}$ and $\CHd=\CHd_{\CV}$ depend on $\CV$, but we continue to omit the annotation.) The lax extension 
${\bf o}\overline{\CHd\CH}{\bf d}=\check{U}$ is the Kleisli extension to $\VRel$ of the up-set functor that has been mentioned in \cite{Hofmann2014} as Exercise IV.2.H. Explicitly, for a $\CV$-relation $r:X\rto Y$,
\[\check{U} r(\CA,\CB)=\bw_{B\in\CB}\bv_{A\in\CA}\bw_{x\in A}\bv_{y\in B}r(x,y)\] for all $\CA\in UX, \CB\in UY$.
The lax extension $\CU\overline{\CH\CHd}\CD=\hat{P}$ gives another lax extension of $P$ to $\VRel$, with
\[\hat{U} r(\CA,\CB)=\bw_{A\in\CA}\bv_{B\in\CB}\bw_{y\in B}\bv_{x\in A}r(x,y)\] for all $\CV$-relations $r:X\rto Y$ and $\CA\in UX, \CB\in UY$.

Finally, the filter monad $\mathbb{F}$ and the ultrafilter monad on $\Set$ being submonads of $\mathbb{U}$, one may obtain the lax extensions $\hat{F}$, $\check{F}$ to $\CV$-${\bf Rel}$, as well the Barr extension of the ultrafilter monad as initial extensions of $\hat{U},\check{U}$. The lax algebras pertaining to the Barr extension to $\CV$-$\bf{Rel}$ of the ultrafilter monad have recently been presented in various forms in \cite{Lai2016}.
\end{exmp}

$$ $$
{\em Acknowledgements.}
The first author acknowledges the support of National Natural Science Foundation of China (11101297). He is also grateful to Walter Tholen for warm hospitality during his visit to the Department of Mathematics and Statistics at York University when this paper was completed.

The second author acknowledges the support of the Natural Sciences and Engineering Research Council of Canada (Discovery Grant 501260).

$$   $$



%\section*{\refname}
%\addcontentsline{toc}{section}{\refname}


\begin{thebibliography}{10}

\bibitem{Akhvlediani2010}
A.~Akhvlediani, M.~M. Clementino, and W.~Tholen.
\newblock On the categorical meaning of Hausdorff and Gromov distances, I.
\newblock {\em Topology and its Applications}, 157(8):1275 -- 1295, 2010.

\bibitem{Barr1970}
M.~Barr.
\newblock Relational algebras.
\newblock {\em Lecture Notes in Mathematics}, 137:39--55.
\newblock Springer, Berlin, 1970.

\bibitem{Clementino2003}
M.~Clementino and D.~Hofmann.
\newblock Topological features of lax algebras.
\newblock {\em Applied Categorical Structures}, 11:267--286, 2003.

\bibitem{Clementino2004}
M.~M. Clementino and D.~Hofmann.
\newblock On extensions of lax monads.
\newblock {\em Theory and  Applications of Categories}, 13:No. 3, 41--60, 2004.

\bibitem{Heymans2010}
H.~Heymans.
\newblock {\em Sheaves on Quantales as Generalized Metric Spaces}.
\newblock PhD thesis, Universiteit Antwerpen, Belgium, 2010.

\bibitem{Hofmann2014}
D.~Hofmann, G.~J. Seal, and W.~Tholen, editors.
\newblock {\em Monoidal Topology: A Categorical Approach to Order, Metric, and
  Topology}, volume 153 of {\em Encyclopedia of Mathematics and its
  Applications}.
\newblock Cambridge University Press, Cambridge, 2014.

\bibitem{Joyal1984}
A.~Joyal and M.~Tierney.
\newblock An extension of the {Galois} theory of {Grothendieck}.
\newblock {\em Memoirs of the American Mathematical Society}, 51(309), 1984.

\bibitem{LST}
H.~Lai, L.~Shen and W.~Tholen.
\newblock Lax distributive laws for topology, II.
\newblock {\em arXiv:1605.04489}, 2016.

\bibitem{Lai2016}
H.~Lai and W.~Tholen.
\newblock Quantale-valued approach spaces via closure and convergence.
\newblock {\em 	arXiv:1604.08813}, 2016.

\bibitem{Manes1969}
E.G.~Manes.
\newblock A triple-theoretic construction of compact algebras. 
\newblock {em. Lecture Notes In Mathematics}, 80:91--118.
\newblock Springer, Berlin, 1969.

\bibitem{Pu2015}
Q.~Pu and D.~Zhang.
\newblock Categories enriched over a quantaloid: Algebras.
\newblock {\em Theory and Applications of Categories}, 30:751--774, 2015.

\bibitem{Rosenthal1996}
K.~I. Rosenthal.
\newblock {\em The Theory of Quantaloids}, volume 348 of {\em Pitman Research
  Notes in Mathematics Series}.
\newblock Longman, Harlow, 1996.

\bibitem{Shen2013a}
L.~Shen and D.~Zhang.
\newblock Categories enriched over a quantaloid: {Isbell} adjunctions and {Kan}
  adjunctions.
\newblock {\em Theory and Applications of Categories}, 28:577--615, 2013.

\bibitem{Stubbe2010}
I.~Stubbe.
\newblock `Hausdorff distance' via conical cocompletion.
\newblock {\em Cahiers de Topologie et G{\'e}om{\'e}trie Diff{\'e}rentielle
  Cat{\'e}goriques}, LI:51--76, 2010.

\bibitem{Stubbe2013}
I.~Stubbe.
\newblock The double power monad is the composite power monad.
\newblock {\em preprint}, 2013.

\bibitem{Tholen2009}
W.~Tholen.
\newblock Ordered topological structures.
\newblock {\em Topology and its Applications}, 156:2148--2157, 2009.

\bibitem{Tholen2016}
W.~Tholen.
\newblock Lax distributive laws for topology, I.
\newblock {\em arXiv:1603.06251}, 2016.

\end{thebibliography}
\end{document}